\documentclass[10pt,a4paper]{article}
\usepackage[utf8]{inputenc}
\usepackage[english]{babel}
\usepackage[T1]{fontenc}

\usepackage{amsmath}
\usepackage{amsthm}
\usepackage{amsfonts}
\usepackage{amssymb}
\usepackage{scrextend}
\usepackage[colorlinks=true, allcolors=blue]{hyperref}
\usepackage{cleveref}
\usepackage{graphicx}
\usepackage{enumerate}
\usepackage{lmodern}
\usepackage{wrapfig}
\usepackage{mathrsfs}
\usepackage{subcaption}
\usepackage[onelanguage,ruled,vlined,linesnumbered]{algorithm2e}
\usepackage{siunitx}
\usepackage{pdfpages}
\usepackage{float}
\usepackage[colorinlistoftodos]{todonotes}
\usepackage{comment}
\usepackage[parfill]{parskip}
\usepackage{bm}
\usepackage{xcolor}
\usepackage{makecell}
\usepackage{colortbl}
\usepackage{tablefootnote}
\usepackage{cancel}

\usepackage{diagbox}

\usepackage{tikz}
\usetikzlibrary{shapes,graphs,graphs.standard}

\newcommand{\db}[2]{#1\!\leftrightarrow\!#2}
\newcommand{\ru}[1]{\textbf{R#1}}

\newtheorem{theorem}{Theorem}

\newtheorem{definition}[theorem]{Definition}

\newtheorem{proposition}[theorem]{Proposition}
\crefname{proposition}{Proposition}{Propositions}
\newtheorem{lemma}[theorem]{Lemma}
\newtheorem{conjecture}[theorem]{Conjecture}

\crefname{claim}{claim}{claims}
\Crefname{claim}{Claim}{Claims}

\usepackage[left=2cm,right=2cm,top=2cm,bottom=2cm]{geometry}

\DeclareMathOperator{\ad}{ad}
\DeclareMathOperator{\mad}{mad}

\usepackage{authblk}

\title{$2$-distance list $(\Delta+2)$-coloring of planar graphs with girth at least 10}

\author[1]{Hoang La\thanks{xuan-hoang.la@lirmm.fr}}
\author[1]{Mickael Montassier\thanks{mickael.montassier@lirmm.fr}}
\affil[1]{LIRMM, Université de Montpellier, CNRS, Montpellier, France}
\begin{document}
  \maketitle

\begin{abstract}
Given a graph $G$ and a list assignment $L(v)$ for each vertex of $v$ of $G$. A proper $L$-list-coloring of $G$ is a function that maps every vertex to a color in $L(v)$ such that no pair of adjacent vertices have the same color. We say that a graph is list $k$-colorable when every vertex $v$ has a list of colors of size at least $k$.
A $2$-distance coloring is a coloring where vertices at distance at most 2 cannot share the same color. We prove the existence of a $2$-distance list ($\Delta+2$)-coloring for planar graphs with girth at least $10$ and maximum degree $\Delta\geq 4$.
\end{abstract}

\section{Introduction}

A \emph{$k$-coloring} of the vertices of a graph $G=(V,E)$ is a map $\phi:V \rightarrow\{1,2,\dots,k\}$. A $k$-coloring $\phi$ is a \emph{proper coloring}, if and only if, for all edge $xy\in E,\phi(x)\neq\phi(y)$. In other words, no two adjacent vertices share the same color. The \emph{chromatic number} of $G$, denoted by $\chi(G)$, is the smallest integer $k$ such that $G$ has a proper $k$-coloring.  A generalization of $k$-coloring is $k$-list-coloring.
A graph $G$ is {\em $L$-list colorable} if for a
given list assignment $L=\{L(v): v\in V(G)\}$ there is a proper
coloring $\phi$ of $G$ such that for all $v \in V(G), \phi(v)\in
L(v)$. If $G$ is $L$-list colorable for every list assignment $L$ with $|L(v)|\ge k$ for all $v\in V(G)$, then $G$ is said to be {\em $k$-choosable} or \emph{$k$-list-colorable}. The \emph{list chromatic number} of a graph $G$ is the smallest integer $k$ such that $G$ is $k$-choosable. List coloring can be very different from usual coloring as there exist graphs with a small chromatic number and an arbitrarily large list chromatic number.

In 1969, Kramer and Kramer introduced the notion of 2-distance coloring \cite{kramer2,kramer1}. This notion generalizes the ``proper'' constraint (that does not allow two adjacent vertices to have the same color) in the following way: a \emph{$2$-distance $k$-coloring} is such that no pair of vertices at distance at most 2 have the same color. The \emph{$2$-distance chromatic number} of $G$, denoted by $\chi^2(G)$, is the smallest integer $k$ such that $G$ has a 2-distance $k$-coloring. Similarly to proper $k$-list-coloring, one can also define \emph{$2$-distance $k$-list-coloring} We denote $\chi^2_l(G)$ the \emph{$2$-distance list chromatic number} of $G$. 

For all $v\in V$, we denote $d_G(v)$ the degree of $v$ in $G$ and by $\Delta(G) = \max_{v\in V}d_G(v)$ the maximum degree of a graph $G$. For brevity, when it is clear from the context, we will use $\Delta$ (resp. $d(v)$) instead of $\Delta(G)$ (resp. $d_G(v)$). 
One can observe that, for any graph $G$, $\Delta+1\leq\chi^2(G)\leq \Delta^2+1$. The lower bound is trivial since, in a 2-distance coloring, every neighbor of a vertex $v$ with degree $\Delta$, and $v$ itself must have a different color. As for the upper bound, a greedy algorithm shows that $\chi^2(G)\leq \Delta^2+1$. Moreover, that upper bound is tight for some graphs, for example, Moore graphs of type $(\Delta,2)$, which are graphs where all vertices have degree $\Delta$, are at distance at most two from each other, and the total number of vertices is $\Delta^2+1$. See \Cref{tight upper bound figure}.

\begin{figure}[htbp]
\begin{center}
\begin{subfigure}[t]{5cm}
\centering
\begin{tikzpicture}[every node/.style={circle,thick,draw,minimum size=1pt,inner sep=2}]
  \graph[clockwise, radius=1.5cm] {subgraph C_n [n=5,name=A] };
\end{tikzpicture}
\caption{The Moore graph of type (2,2):\\ the odd cycle $C_5$.}
\end{subfigure}
\qquad
\begin{subfigure}[t]{5cm}
\centering
\begin{tikzpicture}[every node/.style={circle,thick,draw,minimum size=1pt,inner sep=1}]
  \graph[clockwise, radius=1.5cm] {subgraph C_n [n=5,name=A] };
  \graph[clockwise, radius=0.75cm,n=5,name=B] {1/"6", 2/"7", 3/"8", 4/"9", 5/"10" };

  \foreach \i [evaluate={\j=int(mod(\i+6,5)+1)}]
     in {1,2,3,4,5}{
    \draw (A \i) -- (B \i);
    \draw (B \j) -- (B \i);
  }
\end{tikzpicture}
\caption{The Moore graph of type (3,2):\\ the Petersen graph.}
\end{subfigure}
\qquad
\begin{subfigure}[t]{5cm}
\centering
\includegraphics[scale=0.12]{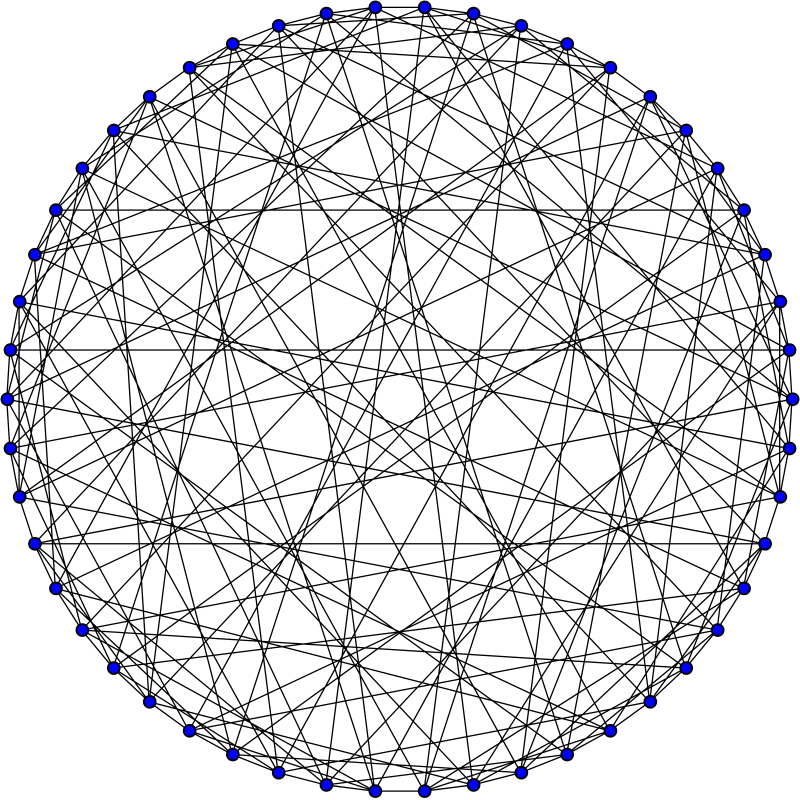}
\caption{The Moore graph of type (7,2):\\ the Hoffman-Singleton graph.}
\end{subfigure}
\caption{Examples of Moore graphs for which $\chi^2=\Delta^2+1$.}
\label{tight upper bound figure}
\end{center}
\end{figure}
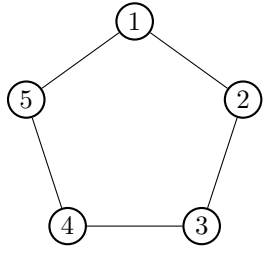
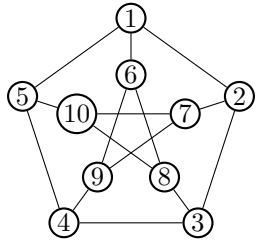
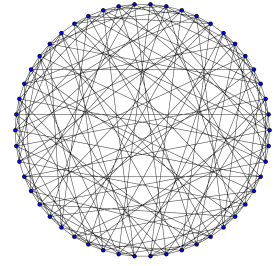

By nature, $2$-distance list colorings and the $2$-distance list chromatic number of a graph depend a lot on the number of vertices in the neighborhood of every vertex. More precisely, the ``sparser'' a graph is, the lower its $2$-distance chromatic number will be. One way to quantify the sparsity of a graph is through its maximum average degree. The \emph{average degree} $\ad$ of a graph $G=(V,E)$ is defined by $\ad(G)=\frac{2|E|}{|V|}$. The \emph{maximum average degree} $\mad(G)$ is the maximum, over all subgraphs $H$ of $G$, of $\ad(H)$. Another way to measure the sparsity is through the girth, i.e. the length of a shortest cycle. We denote $g(G)$ the girth of $G$. Intuitively, the higher the girth of a graph is, the sparser it gets. These two measures can actually be linked directly in the case of planar graphs.

A graph is \emph{planar} if one can draw its vertices with points on the plane, and edges with curves intersecting only at its endpoints. When $G$ is a planar graph, Wegner conjectured in 1977 that  $\chi^2(G)$ becomes linear in $\Delta(G)$:

\begin{conjecture}[Wegner \cite{wegner}]
\label{conj:Wegner}
Let $G$ be a planar graph with maximum degree $\Delta$. Then,
$$
\chi^2(G) \leq \left\{
    \begin{array}{ll}
        7, & \mbox{if } \Delta\leq 3, \\
        \Delta + 5, & \mbox{if } 4\leq \Delta\leq 7,\\
        \left\lfloor\frac{3\Delta}{2}\right\rfloor + 1, & \mbox{if } \Delta\geq 8.
    \end{array}
\right.
$$
\end{conjecture}

The upper bound for the case where $\Delta\geq 8$ is tight (see \Cref{wegner figure}(i)). Recently, the case $\Delta\leq 3$ was proved by Thomassen \cite{tho18}, and by Hartke \textit{et al.} \cite{har16} independently. For $\Delta\geq 8$, Havet \textit{et al.} \cite{havet} proved that the bound is $\frac{3}{2}\Delta(1+o(1))$, where $o(1)$ is as $\Delta\rightarrow\infty$ (this bound holds for 2-distance list-colorings). \Cref{conj:Wegner} is known to be true for some subfamilies of planar graphs, for example $K_4$-minor free graphs \cite{lwz03}.

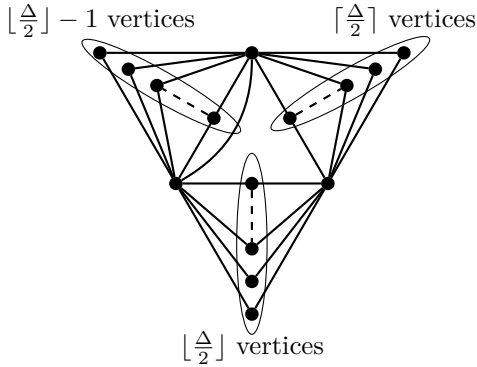
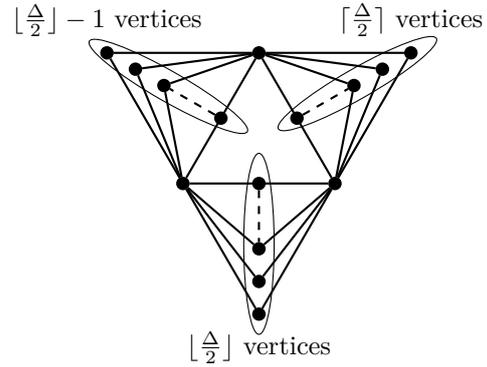
\begin{figure}[htbp]
\begin{subfigure}[b]{0.48\textwidth}
\centering
\begin{tikzpicture}[scale=0.4]
\begin{scope}[every node/.style={circle,thick,draw,minimum size=1pt,inner sep=1.5}]
    \node[fill] (y) at (0,0) {};
    \node[fill] (z) at (5,0) {};
    \node[fill] (x) at (2.5,4.33) {};

    \node[fill] (xy) at (1.25,2.165) {};
    \node[fill] (yz) at (2.5,0) {};
    \node[fill] (zx) at (3.75,2.165) {};

    \node[fill,label={[label distance=-1cm]above:$\lfloor\frac{\Delta}{2}\rfloor-1$ vertices}] (xy1) at (-2.5,4.33) {};
    \node[fill] (xy2) at (-1.5625,3.78875) {};
    \node[fill] (xy3) at (-0.625,3.2475) {};

    \node[fill,label={[label distance=-0.7cm]above:$\lceil\frac{\Delta}{2}\rceil$ vertices}] (zx1) at (7.5,4.33) {};
    \node[fill] (zx2) at (6.5625,3.78875) {};
    \node[fill] (zx3) at (5.625,3.2475) {};

    \node[fill,label={[label distance = -0.7cm]below:$\lfloor\frac{\Delta}{2}\rfloor$ vertices}] (yz1) at (2.5,-4.33) {};
    \node[fill] (yz2) at (2.5,-3.2475) {};
    \node[fill] (yz3) at (2.5,-2.165) {};

\end{scope}

\begin{scope}[every edge/.style={draw=black,thick}]
    \path (x) edge (y);
    \path (y) edge (z);
    \path (z) edge (x);

    \path (x) edge[bend left] (y);

    \path (x) edge (xy1);
    \path (x) edge (xy2);
    \path (x) edge (xy3);

    \path (y) edge (xy1);
    \path (y) edge (xy2);
    \path (y) edge (xy3);

    \path (x) edge (zx1);
    \path (x) edge (zx2);
    \path (x) edge (zx3);

    \path (z) edge (zx1);
    \path (z) edge (zx2);
    \path (z) edge (zx3);

    \path (y) edge (yz1);
    \path (y) edge (yz2);
    \path (y) edge (yz3);

    \path (z) edge (yz1);
    \path (z) edge (yz2);
    \path (z) edge (yz3);

    \path[dashed] (xy) edge (xy3);
    \path[dashed] (yz) edge (yz3);
    \path[dashed] (zx) edge (zx3);
\end{scope}
\draw[rotate=-30] (-0.625-1.4,3.2475-0.7) ellipse (3cm and 0.5cm);
\draw[rotate=30] (5+0.625+1,3.2475-3.25) ellipse (3cm and 0.5cm);
\draw[rotate=90] (-2,-2.5) ellipse (3cm and 0.5cm);
\end{tikzpicture}
\vspace{-0.9cm}
\caption{A graph with girth 3 and $\chi^2=\lfloor\frac{3\Delta}{2}\rfloor+1$.}
\end{subfigure}
\begin{subfigure}[b]{0.48\textwidth}
\centering
\begin{tikzpicture}[scale=0.4]
\begin{scope}[every node/.style={circle,thick,draw,minimum size=1pt,inner sep=1.5}]
    \node[fill] (y) at (0,0) {};
    \node[fill] (z) at (5,0) {};
    \node[fill] (x) at (2.5,4.33) {};

    \node[fill] (xy) at (1.25,2.165) {};
    \node[fill] (yz) at (2.5,0) {};
    \node[fill] (zx) at (3.75,2.165) {};

    \node[fill,label={[label distance=-1cm]above:$\lfloor\frac{\Delta}{2}\rfloor-1$ vertices}] (xy1) at (-2.5,4.33) {};
    \node[fill] (xy2) at (-1.5625,3.78875) {};
    \node[fill] (xy3) at (-0.625,3.2475) {};

    \node[fill,label={[label distance=-0.7cm]above:$\lceil\frac{\Delta}{2}\rceil$ vertices}] (zx1) at (7.5,4.33) {};
    \node[fill] (zx2) at (6.5625,3.78875) {};
    \node[fill] (zx3) at (5.625,3.2475) {};

    \node[fill,label={[label distance = -0.7cm]below:$\lfloor\frac{\Delta}{2}\rfloor$ vertices}] (yz1) at (2.5,-4.33) {};
    \node[fill] (yz2) at (2.5,-3.2475) {};
    \node[fill] (yz3) at (2.5,-2.165) {};

\end{scope}

\begin{scope}[every edge/.style={draw=black,thick}]
    \path (x) edge (y);
    \path (y) edge (z);
    \path (z) edge (x);

    \path (x) edge (xy1);
    \path (x) edge (xy2);
    \path (x) edge (xy3);

    \path (y) edge (xy1);
    \path (y) edge (xy2);
    \path (y) edge (xy3);

    \path (x) edge (zx1);
    \path (x) edge (zx2);
    \path (x) edge (zx3);

    \path (z) edge (zx1);
    \path (z) edge (zx2);
    \path (z) edge (zx3);

    \path (y) edge (yz1);
    \path (y) edge (yz2);
    \path (y) edge (yz3);

    \path (z) edge (yz1);
    \path (z) edge (yz2);
    \path (z) edge (yz3);

    \path[dashed] (xy) edge (xy3);
    \path[dashed] (yz) edge (yz3);
    \path[dashed] (zx) edge (zx3);
\end{scope}
\draw[rotate=-30] (-0.625-1.4,3.2475-0.7) ellipse (3cm and 0.5cm);
\draw[rotate=30] (5+0.625+1,3.2475-3.25) ellipse (3cm and 0.5cm);
\draw[rotate=90] (-2,-2.5) ellipse (3cm and 0.5cm);
\end{tikzpicture}
\vspace{-0.9cm}
\caption{A graph with girth 4 and $\chi^2=\lfloor\frac{3\Delta}{2}\rfloor-1$.}
\end{subfigure}
\caption{Graphs with $\chi^2\approx \frac32 \Delta$ \cite{wegner}.}
\label{wegner figure}
\end{figure}

\begin{figure}[htbp]
\begin{subfigure}[b]{0.24\textwidth}
\centering
\begin{tikzpicture}[scale=0.6]
\begin{scope}[every node/.style={circle,thick,draw,minimum size=1pt,inner sep=1.5}]
    \node[fill] (y) at (0,0) {};
    \node[fill] (z) at (5,0) {};
    \node[fill] (x) at (2.5,4.33) {};
    
    \node[fill] (1) at (2.5,2) {};
    \node[fill] (2) at (2.5,3.17) {};
    \node[fill] (3) at (1.25,1) {};
    \node[fill] (4) at (3.75,1) {};
\end{scope}

\begin{scope}[every edge/.style={draw=black,thick}]
    \path (x) edge (y);
    \path (y) edge (z);
    \path (z) edge (x);
    \path (1) edge (x);
    \path (1) edge (y);
    \path (1) edge (z);
\end{scope}
\end{tikzpicture}
\caption{$\Delta=3$ and $\chi^2\geq 7$.}
\end{subfigure}
\begin{subfigure}[b]{0.24\textwidth}
\centering
\begin{tikzpicture}[scale=0.6]
\begin{scope}[every node/.style={circle,thick,draw,minimum size=1pt,inner sep=1.5}]
    \node[fill] (1) at (0,0) {};
    \node[fill] (2) at (5,0) {};
    \node[fill] (3) at (0,5) {};
    \node[fill] (4) at (5,5) {};
    \node[fill] (5) at (2.5,2.5) {};
    \node[fill] (6) at (2.5,3.75) {};
    \node[fill] (7) at (3.75,2.5) {};
    \node[fill] (8) at (2.5,1.25) {};
    \node[fill] (9) at (1.25,2.5) {};
\end{scope}

\begin{scope}[every edge/.style={draw=black,thick}]
    \path (1) edge (2);
    \path (1) edge (3);
    \path (3) edge (4);
    \path (2) edge (4);
    \path (5) edge (6);
    \path (5) edge (7);
    \path (5) edge (8);
    \path (5) edge (9);
    \path (1) edge (9);
    \path (3) edge (9);
    \path (3) edge (6);
    \path (4) edge (6);
    \path (4) edge (7);
    \path (2) edge (7);
    \path (2) edge (8);
    \path (1) edge (8);
\end{scope}
\end{tikzpicture}
\caption{$\Delta=4$ and $\chi^2\geq 9$.}
\end{subfigure}
\begin{subfigure}[b]{0.24\textwidth}
\centering
\begin{tikzpicture}[scale=0.6]
\begin{scope}[every node/.style={circle,thick,draw,minimum size=1pt,inner sep=1.5}]
    \node[fill] (1) at (0,0) {};
    \node[fill] (2) at (5,0) {};
    \node[fill] (3) at (2.5,4.33) {};
    \node[fill] (4) at (1.5,1) {};
    \node[fill] (5) at (2.5,0.714) {};
    \node[fill] (6) at (3.5,0.429) {};
    \node[fill] (7) at (2.5,2.5) {};
    \node[fill] (8) at (2.5,3.5) {};
    \node[fill] (9) at (2,1.75) {};
    \node[fill] (10) at (3,1.5) {};
\end{scope}

\begin{scope}[every edge/.style={draw=black,thick}]
    \path (1) edge (2);
    \path (1) edge (3);
    \path (2) edge (3);
    \path (1) edge (4);
    \path (2) edge (4);
    \path (1) edge (8);
    \path (3) edge (7);
    \path (4) edge (7);
    \path (6) edge (7);
    \path (1) edge (5);
    \path (8) edge (9);
    \path (10) edge (2);
    \path (8) edge (2);
    \path (9) edge (5);
    \path (9) edge (10);
    \path (5) edge (10);
\end{scope}
\end{tikzpicture}
\caption{$\Delta=5$ and $\chi^2\geq 10$.}
\end{subfigure}
\begin{subfigure}[b]{0.24\textwidth}
\centering
\begin{tikzpicture}[scale=0.6]
\begin{scope}[every node/.style={circle,thick,draw,minimum size=1pt,inner sep=1.5}]
    \node[fill] (1) at (0,0) {};
    \node[fill] (2) at (0,5) {};
    \node (3) at (-3,2.5) {};
    \node[fill] (4) at (-2,2.5) {};
    \node[fill] (5) at (-1,2.5) {};
    \node[fill] (6) at (0,2.5) {};
    \node[fill] (7) at (1,2) {};
    \node[fill] (71) at (1,3) {};
    \node[fill] (8) at (2,2) {};
    \node[fill] (81) at (2,3) {};
    \node[fill] (9) at (3,2) {};
    \node[fill] (91) at (3,3) {};
\end{scope}

\begin{scope}[every edge/.style={draw=black,thick}]
	\path (1) edge (2);
	\path (1) edge[dashed] (3);
	\path (1) edge (4);
	\path (1) edge (5);
	\path (2) edge[dashed] (3);
	\path (2) edge (4);
	\path (2) edge (5);
	\path (1) edge (7);
	\path (1) edge (8);
	\path (1) edge (9);
	\path (7) edge (71);
	\path (8) edge (81);
	\path (9) edge (91);
	\path (2) edge (71);
	\path (2) edge (81);
	\path (2) edge (91);
	\path (71) edge (91);
	\path (7) edge (81);
	\path (81) edge (9); 
\end{scope}
\end{tikzpicture}
\caption{$6\leq \Delta \leq 7$ and \\ $\chi^2\geq \Delta+5$.}
\end{subfigure}
\caption{Constructions by Wegner in \cite{wegner}.}
\label{wegner figure2}
\end{figure}
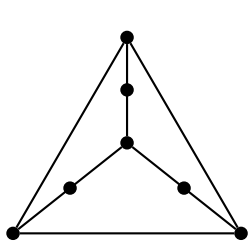
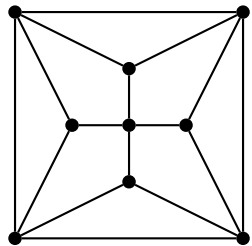
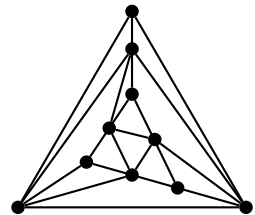
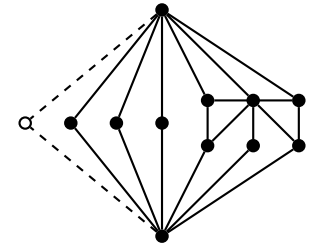

Wegner's conjecture motivated extensive researches on $2$-distance chromatic number of sparse graphs, either of planar graphs with high girth or of graphs with upper bounded maximum average degree which are directly linked due to \Cref{maximum average degree and girth proposition}.

\begin{proposition}[Folklore]\label{maximum average degree and girth proposition}
For every planar graph $G$, $(\mad(G)-2)(g(G)-2)<4$.
\end{proposition}

As a consequence, any theorem with an upper bound on $\mad(G)$ can be translated to a theorem with a lower bound on $g(G)$ under the condition that $G$ is planar. Many results have taken the following form: \textit{every graph $G$ of $\mad(G)\leq m_0$ and $\Delta(G)\geq \Delta_0$ satisfies $\chi^2(G)\leq \Delta(G)+c(m_0,\Delta_0)$ where $c(m_0,\Delta_0)$ is a small constant depending only on $m_0$ and $\Delta_0$}. Due to \Cref{maximum average degree and girth proposition}, as a corollary, we have the same results on planar graphs of girth $g\geq g_0(m_0)$ where $g_0$ depends on $m_0$. \Cref{recap table 2-distance} shows all known such results, up to our knowledge, on the $2$-distance chromatic number of planar graphs with fixed girth, either proven directly for planar graphs with high girth or came as a corollary of a result on graphs with bounded maximum average degree.

\begin{table}[H]
\begin{center}
\scalebox{0.8}{%
\begin{tabular}{||c||c|c|c|c|c|c|c|c||}
\hline
\backslashbox{$g_0$ \kern-1em}{\kern-1em $\chi^2(G)$} & $\Delta+1$ & $\Delta+2$ & $\Delta+3$ & $\Delta+4$ & $\Delta+5$ & $\Delta+6$ & $\Delta+7$ & $\Delta+8$\\
\hline \hline
$3$ & \slashbox{\phantom{\ \ \ \ \ }}{} & \slashbox{\phantom{\ \ \ \ \ }}{} & \slashbox{\phantom{\ \ \ \ \ }}{} & \makecell{$\Delta=3$ \cite{tho18,har16}\\ \cancel{$\Delta\geq 4$}\tablefootnote{\Cref{wegner figure2}} \footref{wegner figure i}}& \cancel{$\Delta\geq 10$}\tablefootnote{\label{wegner figure i}\Cref{wegner figure}(i)} & \cancel{$\Delta\geq 12$}\footref{wegner figure i} & \cancel{$\Delta\geq 14$}\footref{wegner figure i} & \cancel{$\Delta\geq 16$}\footref{wegner figure i} \\
\hline
$4$ & \slashbox{\phantom{\ \ \ \ \ }}{} & \slashbox{\phantom{\ \ \ \ \ }}{} & \cancel{$\Delta\geq 10$}\tablefootnote{\label{wegner figure ii}\Cref{wegner figure}(ii)} & \cancel{$\Delta\geq 12$}\footref{wegner figure ii} & \cancel{$\Delta\geq 14$}\footref{wegner figure ii} & \cancel{$\Delta\geq 16$}\footref{wegner figure ii} & \cancel{$\Delta\geq 18$}\footref{wegner figure ii} & \cancel{$\Delta\geq 20$}\footref{wegner figure ii} \\
\hline
$5$ & \slashbox{\phantom{\ \ \ \ \ }}{} &\makecell{$\Delta\geq 10^7$ \cite{bon19}\footref{list footnote} \\ \cancel{$\Delta=4$}\tablefootnote{\Cref{delta+3 figure}(ii)}} & $\Delta\geq 339$ \cite{don17b} &$\Delta\geq 312$ \cite{don17} & $\Delta\geq 15$ \cite{bu18b}\tablefootnote{\label{other footnote}Corollaries of more general colorings of planar graphs.} &$\Delta\geq 12$ \cite{bu16}\footref{list footnote} & $\Delta\neq 7,8$ \cite{don17} &all $\Delta$ \cite{dl16}\\
\hline
$6$ & \slashbox{\phantom{\ \ \ \ \ }}{} &$\Delta\geq 17$ \cite{bon14}\footref{mad footnote} &$\Delta\geq 9$ \cite{bu16}\footref{list footnote} & &all $\Delta$ \cite{bu11} & & & \\
\hline
$7$ & $\Delta\geq 16$ \cite{iva11}\tablefootnote{\label{list footnote}Corollaries of 2-distance list-colorings of planar graphs.} & $\Delta \geq 10$ \cite{lm21g7d10g8d6} & $\Delta\geq 6$ \cite{la21}\footref{mad list footnote} &$\Delta=4$ \cite{cra13}\tablefootnote{\label{mad list footnote}Corollaries of 2-distance list-colorings of graphs with a bounded maximum average degree.} & & & & \\
\hline
$8$ & $\Delta\geq 9$ \cite{lmpv19}\footref{other footnote}& $\Delta \geq 6$ \cite{lm21g7d10g8d6} & $\Delta\geq 4$ \cite{la21}\footref{mad list footnote} & & & & & \\
\hline
$9$ &$\Delta\geq 7$ \cite{lm21}\footref{mad footnote} &$\Delta=5$ \cite{bu15}\footref{mad list footnote} &$\Delta=3$ \cite{cra07}\footref{list footnote} & & & & & \\
\hline
$10$ & $\Delta\geq 6$ \cite{iva11}\footref{list footnote} & \cellcolor{gray!50} $\Delta \geq 4$\tablefootnote{\label{ours}Corollary of our result.} & & & & & & \\
\hline
$11$ & \cancel{$\Delta = 3$} \cite{lm21g21d3} & $\Delta=4$ \cite{cra13}\footref{mad list footnote} & & & & & & \\
\hline
$12$ & $\Delta=5$ \cite{iva11}\footref{list footnote} &$\Delta=3$ \cite{bi12}\footref{list footnote} & & & & & &\\
\hline
$13$ & & & & & & & &\\
\hline
$14$ &$\Delta\geq 4$ \cite{bon13}\tablefootnote{\label{mad footnote}Corollaries of 2-distance colorings of graphs with a bounded maximum average degree.} & & & & & & &\\
\hline
$\dots$ & & & & & & & & \\
\hline
$21$ & $\Delta=3$ \cite{lm21g21d3} & & & & & & & \\
\hline
\end{tabular}}
\caption{The latest results with a coefficient 1 before $\Delta$ in the upper bound of $\chi^2$.}
\label{recap table 2-distance}
\end{center}
\end{table} 

For example, the result from line ``7'' and column ``$\Delta + 1$'' from \Cref{recap table 2-distance} reads as follows : ``\emph{every planar graph $G$ of girth at least 7  and of $\Delta$ at least 16 satisfies $\chi^2(G)\leq \Delta+1$}''. The crossed out cases in the first column correspond to the fact that, for $g_0\leq 6$, there are planar graphs $G$ with $\chi^2(G)=\Delta+2$ for arbitrarily large $\Delta$ \cite{bor04,dvo08b}. There exists a construction for $\Delta = 3$, girth $11$, and $\chi^2\geq \Delta+2$  detailed in \cite{lm21g21d3}. In \Cref{delta+3 figure}(i), we show a simple planar graph with girth 4 and $\chi^2\geq \Delta+3$ for all $\Delta\geq 2$, which justifies the crossed out cases in the second column. In \Cref{delta+3 figure}(ii), we show a graph with $\Delta = 4$, girth 5, and $\chi^2\geq 7$. For the first square of the third column, \Cref{wegner figure2} shows graphs with $\chi^2\geq \Delta+4$ for $3\leq \Delta\leq 7$ and starting from $\Delta\geq 8$, the graph in \Cref{wegner figure}(i) verifies $\chi^2\geq \Delta + 5$. Similarly, the rest of the crossed out values of $\Delta$ comes from the constructions in \Cref{wegner figure} and \Cref{wegner figure2}.

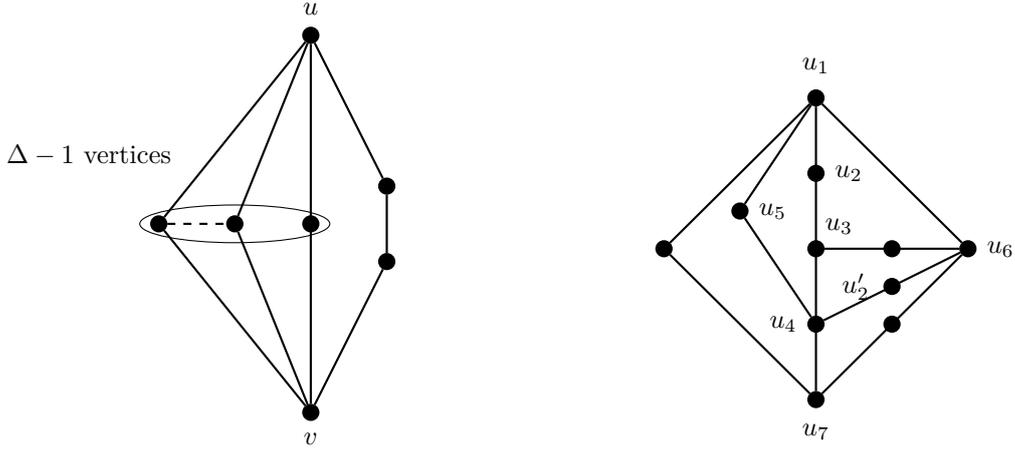
\begin{figure}
\begin{subfigure}[b]{0.49\textwidth}
\centering
\begin{tikzpicture}
\begin{scope}[every node/.style={circle,thick,draw,minimum size=1pt,inner sep=2}]
    \node[fill,label={below:$v$}] (1) at (0,0) {};
    \node[fill,label={above:$u$}] (2) at (0,5) {};
    \node[fill,label={above left:$\Delta - 1$ vertices}] (4) at (-2,2.5) {};
    \node[fill] (5) at (-1,2.5) {};
    \node[fill] (6) at (0,2.5) {};
    \node[fill] (7) at (1,2) {};
    \node[fill] (71) at (1,3) {};
\end{scope}

\begin{scope}[every edge/.style={draw=black,thick}]
	\path (1) edge (2);
	\path (1) edge (4);
	\path (1) edge (5);
	\path (2) edge (4);
	\path (2) edge (5);
	\path (1) edge (7);
	\path (7) edge (71);
	\path (2) edge (71);
	\path (4) edge[dashed] (5);
\end{scope}
\draw (-1,2.5) ellipse (1.25cm and 0.25cm);
\end{tikzpicture}
\caption{A graph with girth 4 and $\chi^2\geq \Delta+3$.}
\end{subfigure}
\begin{subfigure}[b]{0.49\textwidth}
\centering
\begin{tikzpicture}
\begin{scope}[every node/.style={circle,thick,draw,minimum size=1pt,inner sep=2}]
    \node[fill,label={above:$u_1$}] (1) at (0,4) {};
    \node[fill,label={right:$u_2$}] (2) at (0,3) {};
    \node[fill,label={above right:$u_3$}] (3) at (0,2) {};
    \node[fill,label={left:$u_4$}] (4) at (0,1) {};
    \node[fill,label={below:$u_7$}] (5) at (0,0) {};
    \node[fill,label={right:$u_5$}] (6) at (-1,2.5) {};
    \node[fill,label={right:$u_6$}] (7) at (2,2) {};
    \node[fill,label={left:$u'_2$}] (8) at (1,1.5) {};
    \node[fill] (9) at (1,2) {};
    \node[fill] (10) at (-2,2) {};
    \node[fill] (11) at (1,1) {};
\end{scope}

\begin{scope}[every edge/.style={draw=black,thick}]
	\path (1) edge (5);
	\path (1) edge (6);
	\path (6) edge (4);
	\path (1) edge (7);
	\path (7) edge (4);
	\path (3) edge (7);
	\path (1) edge (10);
	\path (10) edge (5);
	\path (7) edge (5);
\end{scope}
\end{tikzpicture}
\caption{A graph with $\Delta=4$, girth 5, and $\chi^2\geq 7$.}
\end{subfigure}
\caption{Graphs with $\chi^2\geq \Delta +3$.}
\label{delta+3 figure}
\end{figure}

We are interested in the case $\chi^2(G)\leq \Delta+2$. More particularly, for a fixed $\Delta_0$, we want to find the lowest value $g_0$ such that planar graphs $G$ with maximum degree $\Delta_0$ and girth at least $g_0$ verify $\chi^2(G)\leq \Delta_0+2$. 

In what follows, we concentrate on the case $\Delta_0= 4$. In \Cref{delta+3 figure}, we provide some simple graphs that give us a lower bound on $g_0$ depending on $\Delta_0$. 

For \Cref{delta+3 figure}(i), $u$, $v$ and their common neighbors must all be colored differently, say they are colored $1$ to $\Delta+1$. The two remaining vertices must be colored with two different colors. As each of them sees every color from $1$ to $\Delta+1$, they must be colored with $\Delta+2$ and $\Delta+3$.

For \Cref{delta+3 figure}(ii), suppose that it is $2$-distance colorable with only six colors. Vertices $u_1$, $u_2$, $u_3$, $u_4$, and $u_5$ form a cycle of length $5$ so they must all be colored differently, say $u_i$ is colored $i$ for $1\leq i\leq 5$. Vertex $u_6$ sees every color from $1$ to $5$ so it must be colored $6$. Vertex $u'_2$ sees every color except $2$ so it must be colored $2$. Finally, $u_7$ sees every color from $1$ to $6$, which is a contradiction.

Now, for the upper bound on $g_0$, we are going to prove that $g_0\leq 10$. In other words,
\begin{theorem}
If $G$ is a planar graph with $g(G)\geq 10$ and $\Delta(G)\geq 4$, then $\chi^2(G)\leq \Delta(G)+2$.
\end{theorem}

As the following results are already known:
\begin{theorem}[Bu \textit{et al.} \cite{bu15}]
If $G$ is a planar graph with $g(G)\geq 9$ and $\Delta(G)=5$, then $\chi^2(G)\leq \Delta(G)+2$.
\end{theorem}

\begin{theorem}[La and Montassier \cite{lm21g7d10g8d6}]
If $G$ is a planar graph with $g(G)\geq 8$ and $\Delta(G)\geq 6$, then $\chi^2(G)\leq \Delta(G)+2$.
\end{theorem}

we only need to prove that:
\begin{theorem}
If $G$ is a planar graph with $g(G)\geq 10$ and $\Delta(G)=4$, then $\chi^2(G)\leq 6$.
\end{theorem}

We will be proving a slightly stronger version which states:
\begin{theorem} \label{main theorem}
If $G$ is a graph with $\mad(G)<\frac52$, $g(G)\geq 10$ and $\Delta(G)=4$, then $\chi^2_l(G)\leq 6$.
\end{theorem}

The latter also improves upon a result from \cite{cra13}.
\begin{theorem}[Cranston \textit{et al.} \cite{cra13}] 
If $G$ is a planar graph with $g(G)\geq 11$ and $\Delta(G)=4$, then $\chi^2_l(G)\leq 6$.
\end{theorem}

In \Cref{sec2}, we present the proof of \Cref{main theorem} using the well-known discharging method.

\section{Proof of \Cref{main theorem}}
\label{sec2}

\paragraph{Notations and drawing conventions.} For $v\in V(G)$, the \emph{2-distance neighborhood} of $v$, denoted $N^*_G(v)$, is the set of 2-distance neighbors of $v$, which are vertices at distance at most two from $v$, not including $v$. We also denote $d^*_G(v)=|N^*_G(v)|$. We will drop the subscript and the argument when it is clear from the context. Also for conciseness, from now on, when we say ``to color'' a vertex, it means to color such vertex differently from all of its colored neighbors at distance at most two. Similarly, any considered coloring will be a 2-distance list-coloring. We will also say that a vertex $u$ ``sees'' another vertex $v$ if $u$ and $v$ are at distance at most 2 from each other. 

Some more notations:
\begin{itemize}
\item A \emph{$d$-vertex} ($d^+$-vertex, $d^-$-vertex) is a vertex of degree $d$ (at least $d$, at most $d$). A \emph{$(\db{d}{e})$-vertex} is a vertex of degree between $d$ and $e$ included.
\item A \emph{$k$-path} ($k^+$-path, $k^-$-path) is a path of length $k+1$ (at least $k+1$, at most $k+1$) where the $k$ internal vertices are 2-vertices and the endvertices are $3^+$-vertices.
\item A \emph{$(k_1,k_2,\dots,k_d)$-vertex} is a $d$-vertex incident to $d$ different paths, where the $i^{\rm th}$ path is a $k_i$-path for all $1\leq i\leq d$.
\end{itemize}

As a drawing convention for the rest of the figures, black vertices will have a fixed degree, which is represented, and white vertices may have a higher degree than what is drawn.

Let $G$ be a counterexample to \Cref{main theorem} with the fewest number of vertices. The purpose of the proof is to prove that $G$ cannot exist. In the following sections, we will study the structural properties of $G$ (\Cref{tutu}), then, we will apply a discharging procedure (\Cref{tonton}).

\subsection{Useful observations}

Before studying the structural properties of $G$, we will introduce some useful observations and lemmas that will be the core of the reducibility proofs of our configurations.

For a vertex $u$, let $L(u)$ denote the set of available colors for $u$. For convenience, the lower bound on $|L(u)|$ will be depicted on the figures below the corresponding vertex $u$. 

\begin{lemma} \label{colorable lemma}
Every graph with list assignment $L$ depicted in \Cref{colorable figure} is $L$-list-colorable.
\end{lemma}

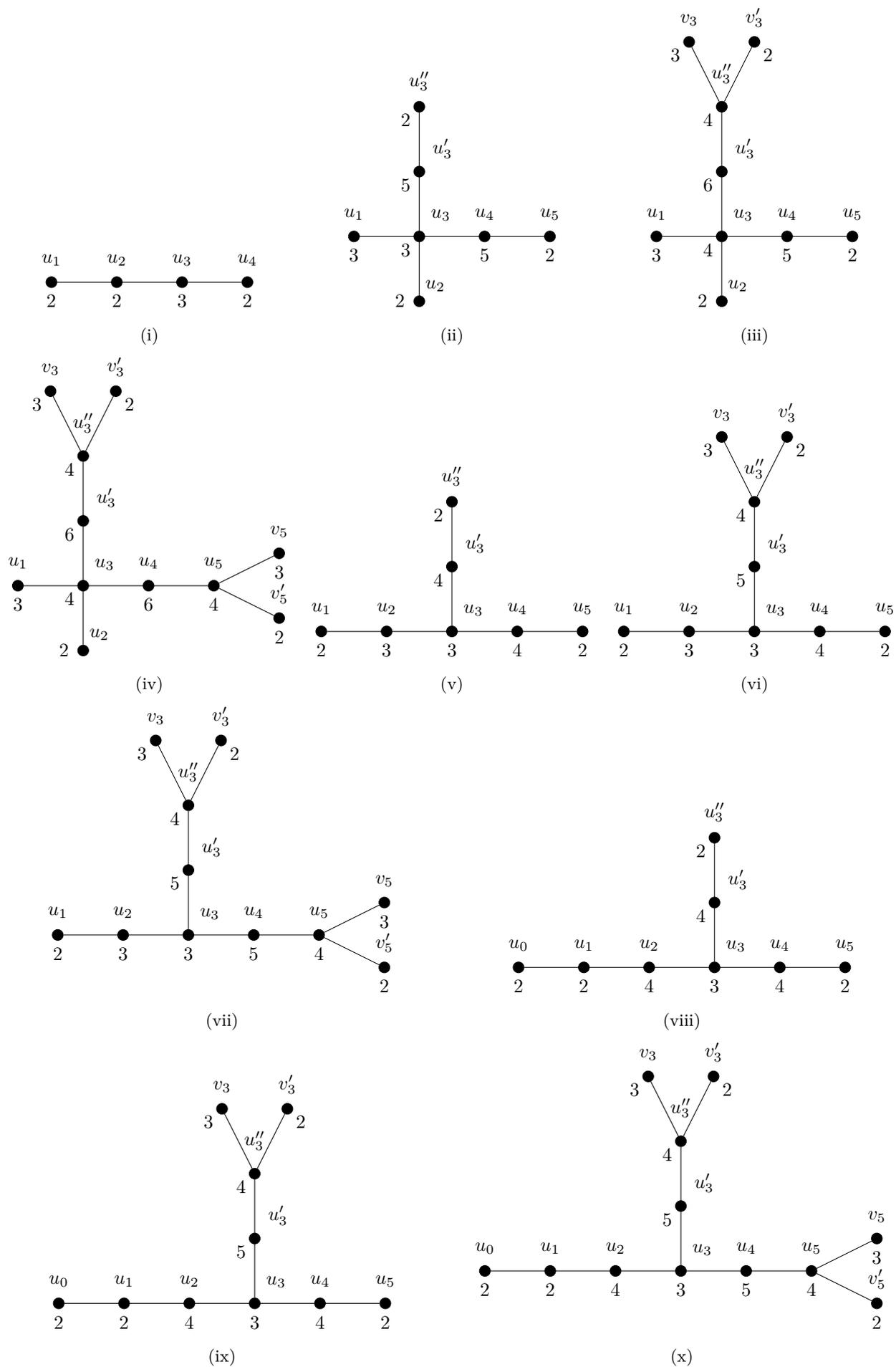
\begin{figure}[!htbp]
\begin{center}
\begin{subfigure}[b]{0.32\textwidth}
\centering
\begin{tikzpicture}[scale=0.6]{thick}
\begin{scope}[every node/.style={circle,draw,minimum size=1pt,inner sep=2}]
    \node[fill,label={above:$u_1$},label={below:$2$}] (1) at (0,0) {};
    \node[fill,label={above:$u_2$},label={below:$2$}] (2) at (2,0) {};
    \node[fill,label={above:$u_3$},label={below:$3$}] (3) at (4,0) {};
    \node[fill,label={above:$u_4$},label={below:$2$}] (4) at (6,0) {};
\end{scope}

\begin{scope}[every edge/.style={draw=black}]
    \path (1) edge (4);
\end{scope}
\end{tikzpicture}
\caption{\label{1}}
\end{subfigure}
\begin{subfigure}[b]{0.32\textwidth}
\centering
\begin{tikzpicture}[scale=0.6]{thick}
\begin{scope}[every node/.style={circle,draw,minimum size=1pt,inner sep=2}]
    \node[fill,label={above:$u_1$},label={below:$3$}] (2) at (2,0) {};
    \node[fill,label={[label distance = +4pt]above right:$u_3$},label={below left:$3$}] (3) at (4,0) {};
    \node[fill,label={above:$u_4$},label={below:$5$}] (4) at (6,0) {};
    \node[fill,label={above:$u_5$},label={below:$2$}] (5) at (8,0) {};
    
    \node[fill,label={[label distance = +4pt]above right:$u'_3$},label={below left:$5$}] (3') at (4,2) {};
    \node[fill,label={above:$u''_3$},label={below left:$2$}] (3'') at (4,4) {};
    
    \node[fill,label={above right:$u_2$},label={left:$2$}] (3''') at (4,-2) {};
\end{scope}

\begin{scope}[every edge/.style={draw=black}]
    \path (2) edge (5);
    \path (3'') edge (3''');
\end{scope}
\end{tikzpicture}
\caption{\label{2}}
\end{subfigure}
\begin{subfigure}[b]{0.32\textwidth}
\centering
\begin{tikzpicture}[scale=0.6]{thick}
\begin{scope}[every node/.style={circle,draw,minimum size=1pt,inner sep=2}]
    \node[fill,label={above:$u_1$},label={below:$3$}] (2) at (2,0) {};
    \node[fill,label={[label distance = +4pt]above right:$u_3$},label={below left:$4$}] (3) at (4,0) {};
    \node[fill,label={above:$u_4$},label={below:$5$}] (4) at (6,0) {};
    \node[fill,label={above:$u_5$},label={below:$2$}] (5) at (8,0) {};
	    
    \node[fill,label={above right:$u_2$},label={left:$2$}] (3''') at (4,-2) {};
    
    \node[fill,label={[label distance = +4pt]above right:$u'_3$},label={below left:$6$}] (3') at (4,2) {};
    \node[fill,label={[label distance = +4pt]above:$u''_3$},label={below left:$4$}] (3'') at (4,4) {};
    \node[fill,label={above:$v_3$},label={below left:$3$}] (v3) at (3,6) {};
    \node[fill,label={above:$v'_3$},label={below right:$2$}] (v'3) at (5,6) {};
\end{scope}

\begin{scope}[every edge/.style={draw=black}]
    \path (2) edge (5);
    \path (3''') edge (3'');
    \path (3'') edge (v3);
    \path (3'') edge (v'3);
\end{scope}
\end{tikzpicture}
\caption{\label{3}}
\end{subfigure}
\begin{subfigure}[b]{0.32\textwidth}
\centering
\begin{tikzpicture}[scale=0.6]{thick}
\begin{scope}[every node/.style={circle,draw,minimum size=1pt,inner sep=2}]
    \node[fill,label={above:$u_1$},label={below:$3$}] (2) at (2,0) {};
    \node[fill,label={[label distance = +4pt]above right:$u_3$},label={below left:$4$}] (3) at (4,0) {};
    \node[fill,label={above:$u_4$},label={below:$6$}] (4) at (6,0) {};
    \node[fill,label={above:$u_5$},label={below:$4$}] (5) at (8,0) {};
	    
    \node[fill,label={above right:$u_2$},label={left:$2$}] (3''') at (4,-2) {};
    
    \node[fill,label={[label distance = +4pt]above right:$u'_3$},label={below left:$6$}] (3') at (4,2) {};
    \node[fill,label={[label distance = +4pt]above:$u''_3$},label={below left:$4$}] (3'') at (4,4) {};
    \node[fill,label={above:$v_3$},label={below left:$3$}] (v3) at (3,6) {};
    \node[fill,label={above:$v'_3$},label={below right:$2$}] (v'3) at (5,6) {};
    
    \node[fill,label={above:$v_5$},label={below:$3$}] (v5) at (10,1) {};
    \node[fill,label={above:$v'_5$},label={below:$2$}] (v'5) at (10,-1) {};
\end{scope}

\begin{scope}[every edge/.style={draw=black}]
    \path (2) edge (5);
    \path (3''') edge (3'');
    \path (3'') edge (v3);
    \path (3'') edge (v'3);
    \path (5) edge (v5);
    \path (5) edge (v'5);
\end{scope}
\end{tikzpicture}
\caption{\label{4}}
\end{subfigure}
\begin{subfigure}[b]{0.32\textwidth}
\centering
\begin{tikzpicture}[scale=0.6]{thick}
\begin{scope}[every node/.style={circle,draw,minimum size=1pt,inner sep=2}]
    \node[fill,label={above:$u_1$},label={below:$2$}] (1) at (0,0) {};
    \node[fill,label={above:$u_2$},label={below:$3$}] (2) at (2,0) {};
    \node[fill,label={[label distance = +4pt]above right:$u_3$},label={below:$3$}] (3) at (4,0) {};
    \node[fill,label={above:$u_4$},label={below:$4$}] (4) at (6,0) {};
    \node[fill,label={above:$u_5$},label={below:$2$}] (5) at (8,0) {};
    
    \node[fill,label={[label distance = +4pt]above right:$u'_3$},label={below left:$4$}] (3') at (4,2) {};
    \node[fill,label={above:$u''_3$},label={below left:$2$}] (3'') at (4,4) {};
\end{scope}

\begin{scope}[every edge/.style={draw=black}]
    \path (1) edge (5);
    \path (3) edge (3'');
\end{scope}
\end{tikzpicture}
\caption{\label{5}}
\end{subfigure}
\begin{subfigure}[b]{0.32\textwidth}
\centering
\begin{tikzpicture}[scale=0.6]{thick}
\begin{scope}[every node/.style={circle,draw,minimum size=1pt,inner sep=2}]
    \node[fill,label={above:$u_1$},label={below:$2$}] (1) at (0,0) {};
    \node[fill,label={above:$u_2$},label={below:$3$}] (2) at (2,0) {};
    \node[fill,label={[label distance = +4pt]above right:$u_3$},label={below:$3$}] (3) at (4,0) {};
    \node[fill,label={above:$u_4$},label={below:$4$}] (4) at (6,0) {};
    \node[fill,label={above:$u_5$},label={below:$2$}] (5) at (8,0) {};
    
    \node[fill,label={[label distance = +4pt]above right:$u'_3$},label={below left:$5$}] (3') at (4,2) {};
    \node[fill,label={[label distance = +4pt]above:$u''_3$},label={below left:$4$}] (3'') at (4,4) {};
    \node[fill,label={above:$v_3$},label={below left:$3$}] (v3) at (3,6) {};
    \node[fill,label={above:$v'_3$},label={below right:$2$}] (v'3) at (5,6) {};
\end{scope}

\begin{scope}[every edge/.style={draw=black}]
    \path (1) edge (5);
    \path (3) edge (3'');
    \path (3'') edge (v3);
    \path (3'') edge (v'3);
\end{scope}
\end{tikzpicture}
\caption{\label{6}}
\end{subfigure}
\begin{subfigure}[b]{0.49\textwidth}
\centering
\begin{tikzpicture}[scale=0.6]{thick}
\begin{scope}[every node/.style={circle,draw,minimum size=1pt,inner sep=2}]
    \node[fill,label={above:$u_1$},label={below:$2$}] (1) at (0,0) {};
    \node[fill,label={above:$u_2$},label={below:$3$}] (2) at (2,0) {};
    \node[fill,label={[label distance = +4pt]above right:$u_3$},label={below:$3$}] (3) at (4,0) {};
    \node[fill,label={above:$u_4$},label={below:$5$}] (4) at (6,0) {};
    \node[fill,label={above:$u_5$},label={below:$4$}] (5) at (8,0) {};
    
    \node[fill,label={[label distance = +4pt]above right:$u'_3$},label={below left:$5$}] (3') at (4,2) {};
    \node[fill,label={[label distance = +4pt]above:$u''_3$},label={below left:$4$}] (3'') at (4,4) {};
    \node[fill,label={above:$v_3$},label={below left:$3$}] (v3) at (3,6) {};
    \node[fill,label={above:$v'_3$},label={below right:$2$}] (v'3) at (5,6) {};
    
    \node[fill,label={above:$v_5$},label={below:$3$}] (v5) at (10,1) {};
    \node[fill,label={above:$v'_5$},label={below:$2$}] (v'5) at (10,-1) {};
\end{scope}

\begin{scope}[every edge/.style={draw=black}]
    \path (1) edge (5);
    \path (3) edge (3'');
    \path (3'') edge (v3);
    \path (3'') edge (v'3);
    \path (5) edge (v5);
    \path (5) edge (v'5);
\end{scope}
\end{tikzpicture}
\caption{\label{7}}
\end{subfigure}
\begin{subfigure}[b]{0.49\textwidth}
\centering
\begin{tikzpicture}[scale=0.6]{thick}
\begin{scope}[every node/.style={circle,draw,minimum size=1pt,inner sep=2}]
	\node[fill,label={above:$u_0$},label={below:$2$}] (0) at (-2,0) {};
    \node[fill,label={above:$u_1$},label={below:$2$}] (1) at (0,0) {};
    \node[fill,label={above:$u_2$},label={below:$4$}] (2) at (2,0) {};
    \node[fill,label={[label distance = +4pt]above right:$u_3$},label={below:$3$}] (3) at (4,0) {};
    \node[fill,label={above:$u_4$},label={below:$4$}] (4) at (6,0) {};
    \node[fill,label={above:$u_5$},label={below:$2$}] (5) at (8,0) {};
    
    \node[fill,label={[label distance = +4pt]above right:$u'_3$},label={below left:$4$}] (3') at (4,2) {};
    \node[fill,label={above:$u''_3$},label={below left:$2$}] (3'') at (4,4) {};
\end{scope}

\begin{scope}[every edge/.style={draw=black}]
    \path (0) edge (5);
    \path (3) edge (3'');
\end{scope}
\end{tikzpicture}
\caption{\label{8}}
\end{subfigure}
\begin{subfigure}[b]{0.49\textwidth}
\centering
\begin{tikzpicture}[scale=0.6]{thick}
\begin{scope}[every node/.style={circle,draw,minimum size=1pt,inner sep=2}]			\node[fill,label={above:$u_0$},label={below:$2$}] (0) at (-2,0) {};
    \node[fill,label={above:$u_1$},label={below:$2$}] (1) at (0,0) {};
    \node[fill,label={above:$u_2$},label={below:$4$}] (2) at (2,0) {};
    \node[fill,label={[label distance = +4pt]above right:$u_3$},label={below:$3$}] (3) at (4,0) {};
    \node[fill,label={above:$u_4$},label={below:$4$}] (4) at (6,0) {};
    \node[fill,label={above:$u_5$},label={below:$2$}] (5) at (8,0) {};
    
    \node[fill,label={[label distance = +4pt]above right:$u'_3$},label={below left:$5$}] (3') at (4,2) {};
    \node[fill,label={[label distance = +4pt]above:$u''_3$},label={below left:$4$}] (3'') at (4,4) {};
    \node[fill,label={above:$v_3$},label={below left:$3$}] (v3) at (3,6) {};
    \node[fill,label={above:$v'_3$},label={below right:$2$}] (v'3) at (5,6) {};
\end{scope}

\begin{scope}[every edge/.style={draw=black}]
    \path (0) edge (5);
    \path (3) edge (3'');
    \path (3'') edge (v3);
    \path (3'') edge (v'3);
\end{scope}
\end{tikzpicture}
\caption{\label{9}}
\end{subfigure}
\begin{subfigure}[b]{0.49\textwidth}
\centering
\begin{tikzpicture}[scale=0.6]{thick}
\begin{scope}[every node/.style={circle,draw,minimum size=1pt,inner sep=2}]
	\node[fill,label={above:$u_0$},label={below:$2$}] (0) at (-2,0) {};
    \node[fill,label={above:$u_1$},label={below:$2$}] (1) at (0,0) {};
    \node[fill,label={above:$u_2$},label={below:$4$}] (2) at (2,0) {};
    \node[fill,label={[label distance = +4pt]above right:$u_3$},label={below:$3$}] (3) at (4,0) {};
    \node[fill,label={above:$u_4$},label={below:$5$}] (4) at (6,0) {};
    \node[fill,label={above:$u_5$},label={below:$4$}] (5) at (8,0) {};
    
    \node[fill,label={[label distance = +4pt]above right:$u'_3$},label={below left:$5$}] (3') at (4,2) {};
    \node[fill,label={[label distance = +4pt]above:$u''_3$},label={below left:$4$}] (3'') at (4,4) {};
    \node[fill,label={above:$v_3$},label={below left:$3$}] (v3) at (3,6) {};
    \node[fill,label={above:$v'_3$},label={below right:$2$}] (v'3) at (5,6) {};
    
    \node[fill,label={above:$v_5$},label={below:$3$}] (v5) at (10,1) {};
    \node[fill,label={above:$v'_5$},label={below:$2$}] (v'5) at (10,-1) {};
\end{scope}

\begin{scope}[every edge/.style={draw=black}]
    \path (0) edge (5);
    \path (3) edge (3'');
    \path (3'') edge (v3);
    \path (3'') edge (v'3);
    \path (5) edge (v5);
    \path (5) edge (v'5);
\end{scope}
\end{tikzpicture}
\caption{\label{10}}
\end{subfigure}
\caption{List-colorable graphs.}
\label{colorable figure}
\end{center}
\end{figure}

\begin{proof}
In the following proofs, whenever the size of a list $|L(u)|\geq i$, we assume that $|L(u)|=i$ by removing the extra colors from the list.

\begin{itemize}
\item[(i)] If $L(u_1)=L(u_2)$, then we color $u_3$ with a color in $L(u_3)\setminus L(u_2)$, followed by $u_4$, $u_2$, and $u_1$ in this order. If $L(u_1)\neq L(u_2)$, then we color $u_2$ with a color in $L(u_2)\setminus L(u_1)$, followed by $u_4$, $u_3$, and $u_1$ in this order.

\item[(ii)] First, we claim the following:
\begin{itemize}
\item $L(u_5)\cap L(u_1)=\emptyset$ and $L(u_5)\cap L(u_2)=\emptyset$. Suppose by contradiction that there exists $x\in L(u_5)\cap L(u_1)$, we color $u_1$ and $u_5$ with $x$, then $u_2$, $u_3$, $u''_3$, $u'_3$, and $u_4$ in this order. The same argument holds for $L(u_5)\cap L(u_2)$.
\item $L(u_5)\cap L(u'_3)=\emptyset$. Suppose by contradiction that there exists $x\in L(u_5)\cap L(u'_3)$, we color $u'_3$ and $u_5$ with $x$. Observe that $L(u_5)\cap L(u_1)=\emptyset$ and $L(u_5)\cap L(u_2)=\emptyset$. So, we color $u''_3$, $u_3$, $u_2$, $u_1$, and $u_4$ in this order. 
\item $L(u_5)\cap L(u_3)=\emptyset$. Otherwise, we color $u_3$ with $x\in L(u_5)\cap L(u_3)$. Observe that $L(u_5)\cap L(u_1)=\emptyset$, $L(u_5)\cap L(u_2)=\emptyset$, and $L(u_5)\cap L(u'_3)=\emptyset$. So, we color $u_5$, $u''_3$, $u_2$, $u_1$, $u_4$, and $u'_3$ in this order.
\end{itemize}
Since $L(u_5)\cap L(u_3)=\emptyset$, we color $u_2$, $u_1$, $u_3$, $u''_3$, $u'_3$, $u_4$, and $u_5$ in this order.

\item[(iii)] First, we claim that $L(v'_3)\subset L(v_3)$. Otherwise, we color $v'_3$ with $x\in L(v'_3)\setminus L(v_3)$. Then, we color everything else except $v_3$ thanks to \Cref{2}. We finish by coloring $v_3$.

Now, we color $u''_3$ with $x\in L(u''_3)\setminus L(v_3)$. Since $L(v'_3)\subset L(v_3)$, $x\notin L(v'_3)$. Thus, we color $u_2$, $u_1$, $u_3$, $u_5$, $u_4$, $u'_3$, $v'_3$, and $v_3$ in this order.

\item[(iv)] First, we claim that $L(v'_3)\subset L(v_3)$ and $L(v'_5)\subset L(v_5)$. Suppose by contradiction that there exists $x\in L(v'_5)\setminus L(v_5)$. We color $v'_5$ with $x$. Then, we color everything else except $v_5$ thanks to \Cref{3}. We finish by coloring $v_5$. Symmetrically, the same holds for $L(v'_3)\subset L(v_3)$.

Now, we color $u''_3$ with $x\in L(u''_3)\setminus L(v_3)$. Since $L(v'_3)\subset L(v_3)$, $x\notin L(v'_3)$. Similarly, we color $u_5$ with $y\in L(u_5)\setminus L(v_5)$. We finish by coloring $u_3$, $u_2$, $u_1$, $u'_3$, $v'_3$, $v_3$, $u_4$, $v'_5$, and $v_5$ in this order.

\item[(v)] First, we claim the following:
\begin{itemize}
\item $L(u_2)\cap L(u_5)=\emptyset$ and $L(u_2)\cap L(u''_3)=\emptyset$. Suppose by contradiction that there exists $x\in L(u_2)\cap L(u_5)$. We color $u_2$ and $u_5$ with $x$, then $u_1$, $u_3$, $u''_3$, $u'_3$, and $v_4$ in this order. Symmetrically, the same holds for $L(u_2)\cap L(u''_3)$. 
\item $L(u_3)\subset (L(u_5)\cup L(u''_3))$. Suppose by contradiction that there exists $x\in L(u_3)\setminus (L(u_5)\cup L(u''_3))$. We color $u_3$ with $x$, then $u_1$, $u_2$, $u'_3$, $u''_3$, $u_4$, then $u_5$ in this order.
\item $L(u'_3)\cap L(u_5)=\emptyset$ and $L(u_4)\cap L(u''_3)=\emptyset$. Suppose by contradiction that there exists $x\in L(u'_3)\cap L(u_5)$. We color $u'_3$ and $u_5$ with $x$. Observe that $x\notin L(u_2)$ since $L(u_2)\cap L(u_5)=\emptyset$. Thus, we color $u''_3$, $u_3$, $u_1$, $u_2$ and $u_4$ in this order. Symmetrically, the same holds for $L(u_4)\cap L(u''_3)$. 
\end{itemize}
Since $L(u_3)\subset (L(u_5)\cup L(u''_3))$, $|L(u_3)|=3$, and $|L(u_5)|=|L(u''_3)|=2$, there must exist $x\in L(u_5)\cap L(u_3)$. In addition, $x\notin L(u'_3)$ as $L(u'_3)\cap L(u_5)=\emptyset$. Thus, we color $u_3$ with $x$, then $u_1$, $u_2$, $u_5$, $u_4$, $u''_3$, and $u'_3$ in this order.

\item[(vi)] First, we claim that $L(v'_3)\subset L(v_3)$. Otherwise, we color $v'_3$ with $x\in L(v'_3)\setminus L(v_3)$. Then, we color everything else except $v_3$ thanks to \Cref{5}. We finish by coloring $v_3$.

Now, we color $u''_3$ with $x\in L(u''_3)\setminus L(v_3)$. Since $L(v'_3)\subset L(v_3)$, $x\notin L(v'_3)$. Thus, we color $u_1$, $u_3$, $u_2$, $u_5$, $u_4$, $u'_3$, $v'_3$, and $v_3$ in this order.

\item[(vii)] First, we claim that $L(v'_3)\subset L(v_3)$ and $L(v'_5)\subset L(v_5)$. Suppose by contradiction that there exists $x\in L(v'_5)\setminus L(v_5)$. We color $v'_5$ with $x$. Then, we color everything else except $v_5$ thanks to \Cref{6}. We finish by coloring $v_5$. Symmetrically, the same holds for $L(v'_3)\subset L(v_3)$.

Now, we color $u''_3$ with $x\in L(u''_3)\setminus L(v_3)$. Since $L(v'_3)\subset L(v_3)$, $x\notin L(v'_3)$. Similarly, we color $u_5$ with $y\in L(u_5)\setminus L(v_5)$. We finish by coloring $u_3$, $u_1$, $u_2$, $u'_3$, $v'_3$, $v_3$, $u_4$, $v'_5$, and $v_5$ in this order.

\item[(viii)] First, we claim that $L(u_0) = L(u_1)$. Otherwise, we color $u_0$ with $x\in L(u_0)\setminus L(u_1)$. Then, we color the rest thanks to \Cref{5}.

Since $L(u_0)=L(u_1)$, we can restrict $L(u_2)$ to $L'(u_2)=L(u_2)\setminus L(u_1)$ and observe that, if we can color everything (where $u_2$ has list $L'(u_2)$) except $u_0$ and $u_1$, then we can always finish by coloring $u_1$ and $u_0$ in this order.

Let us show that everything except $u_0$ and $u_1$ can be colored first with the new list $L'(u_2)$ for $u_2$. We claim the following:
\begin{itemize}
\item $L(u_5)\cap L'(u_2)=\emptyset$. Otherwise, we color $u_2$ and $u_5$ with $x\in L(u_5)\cap L'(u_2)$. Then, we color $u_3$, $u''_3$, $u'_3$, and $u_4$ in this order.
\item $L(u_5)\cap L(u'_3)=\emptyset$. Otherwise, we color $u'_3$ and $u_5$ with $x\in L(u_5)\cap L(u'_3)$. Observe that $L(u_5)\cap L'(u_2) =\emptyset$ so $x\notin L'(u_2)$. So, we color $u''_3$, $u_3$, $u_2$, and $u_4$ in this order.
\item $L(u_5)\cap L(u_3)=\emptyset$. Otherwise, we color $u_3$ with $x\in L(u_5)\cap L(u_3)$. Observe that $L(u_5)\cap L'(u_2) =\emptyset$ and $L(u_5)\cap L(u'_3)=\emptyset$ so $x\notin L'(u_2)\cup L(u'_3)$. So, we color $u_5$, $u''_3$, $u_2$, $u_4$, and $u'_3$ in this order.
\end{itemize}
Since $L(u_5)\cap L(u_3)=\emptyset$, we color $u_2$, $u_3$, $u''_3$, $u'_3$, $u_4$, and $u_5$ in this order.
 
\item[(ix)] First, we claim that $L(v'_3)\subset L(v_3)$. Otherwise, we color $v'_3$ with $x\in L(v'_3)\setminus L(v_3)$. Then, we color eveything else except $v_3$ thanks to \Cref{8}. We finish by coloring $v_3$.

Now, we color $u''_3$ with $x\in L(u''_3)\setminus L(v_3)$. Since $L(v'_3)\subset L(v_3)$, $x\notin L(v'_3)$. Thus, we color $u_0$, $u_1$, $u_3$, $u_2$, $u_5$, $u_4$, $u'_3$, $v'_3$, and $v_3$ in this order.

\item[(x)] First, we claim that $L(v'_3)\subset L(v_3)$ and $L(v'_5)\subset L(v_5)$. Suppose by contradiction that there exists $x\in L(v'_5)\setminus L(v_5)$. We color $v'_5$ with $x$. Then, we color everything else except $v_5$ thanks to \Cref{9}. We finish by coloring $v_5$. Symmetrically, the same holds for $L(v'_3)\subset L(v_3)$.

Now, we color $u''_3$ with $x\in L(u''_3)\setminus L(v_3)$. Since $L(v'_3)\subset L(v_3)$, $x\notin L(v'_3)$. Similarly, we color $u_5$ with $y\in L(u_5)\setminus (L(v_5)\cup L(v'_5))$. We finish by coloring $u_3$, $u_1$, $u_0$, $u_2$, $u'_3$, $v'_3$, $v_3$, $u_4$, $v'_5$, and $v_5$ in this order.

\end{itemize}
 
\end{proof}

\subsection{Structural properties of $G$\label{tutu}}

\begin{lemma}\label{connected}
Graph $G$ is connected.
\end{lemma}

\begin{proof}
Otherwise a component of $G$ would be a smaller counterexample.
\end{proof}

\begin{lemma}\label{minimumDegree}
The minimum degree of $G$ is at least 2.
\end{lemma}

\begin{proof}
By \Cref{connected}, the minimum degree is at least 1 or $G$ would be a single isolated vertex contradicting $\Delta(G)=4$. If $G$ contains a degree 1 vertex $v$, then we can simply remove $v$ and 2-distance color the resulting graph, which is possible by minimality of $G$. Then, we add $v$ back and extend the coloring (at most $4$ constraints and $6$ colors).
\end{proof}

\begin{figure}[H]
\begin{subfigure}[b]{0.49\textwidth}
\centering
\begin{tikzpicture}[scale=0.6]{thick}
\begin{scope}[every node/.style={circle,draw,minimum size=1pt,inner sep=2}]
	\node (0) at (-2,0) {};
    \node[fill,label={above:$u$},label={below:$2$}] (1) at (0,0) {};
    \node[fill,label={above:$v$},label={below:$4$}] (2) at (2,0) {};
    \node[fill,label={above:$w$},label={below:$2$}] (3) at (4,0) {};
    \node (4) at (6,0) {};
\end{scope}

\begin{scope}[every edge/.style={draw=black}]
    \path (0) edge (4);
\end{scope}
\end{tikzpicture}
\caption{\label{subfig:3-path}A $3^+$-path.}
\end{subfigure}
\begin{subfigure}[b]{0.49\textwidth}
\centering
\begin{tikzpicture}[scale=0.6]{thick}
\begin{scope}[every node/.style={circle,draw,minimum size=1pt,inner sep=2}]
	\node[label={above:$u$}] (0) at (-2,0) {3};
    \node[fill,label={above:$v$},label={below:$2$}] (1) at (0,0) {};
    \node[fill,label={above:$w$},label={below:$1$}] (2) at (2,0) {};
    \node[label={above:$x$}] (3) at (4,0) {};
\end{scope}

\begin{scope}[every edge/.style={draw=black}]
    \path (0) edge (3);
\end{scope}
\end{tikzpicture}
\caption{\label{subfig:2-path}A $2$-path incident to a $3$-vertex.}
\end{subfigure}
\caption{Path cases.}
\end{figure}
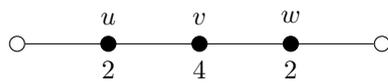
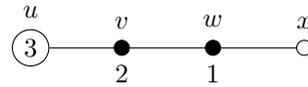

\begin{lemma}\label{3-path lemma}
Graph $G$ does not contain any $3^+$-path.
\end{lemma}

\begin{proof}
Suppose by contradiction that $G$ does contain three consecutives $2$-vertices $uvw$ (see \Cref{subfig:3-path}). It suffices to color $G-\{u,v,w\}$ by minimality of $G$, then we can extend the coloring to the remaining vertices by coloring $u$, $w$, then $v$ in this order. This is possible since $u$, $v$, and $w$ have respectively at least 2, 4, and 2 colors left available.
\end{proof}

\begin{lemma}\label{2-path lemma}
An endvertex of a $2$-path must be a $4$-vertex.
\end{lemma}

\begin{proof}
Suppose by contradiction that there exists a $2$-path $uvwx$ where $d(u)=3$ (see \Cref{subfig:2-path}). We color $G-\{v,w\}$ by minimality of $G$. Then, we color $w$ and $v$ in this order since they have respectively at least 1 and 2 colors left available.
\end{proof}

Let us define some nomenclatures.
\begin{definition}
Let $u$ be a $(1,1,1)$-vertex and let $v$, $w$, and $x$ be the other endvertices of the $1$-paths incident to $u$. We call $u$
\begin{itemize}
\item a \emph{small} $(1,1,1)$-vertex, if $v$, $w$, and $x$ are all $3$-vertices.
\item a \emph{medium} $(1,1,1)$-vertex, if exactly one of $v$, $w$, and $x$ is a $4$-vertex.
\item a \emph{large} $(1,1,1)$-vertex, if exactly two of $v$, $w$, and $x$ are $4$-vertices.
\item a \emph{huge} $(1,1,1)$-vertex, if $v$, $w$, and $x$ are all $4$-vertices.
\end{itemize}
\end{definition}

\begin{definition} \label{special 110}
Let $u$ be a $(1,1,0)$-vertex with a $3$-neighbor and let $u$ share a common $2$-neighbor with a small $(1,1,1)$-vertex. We call $u$ a \emph{special $(1,1,0)$-vertex}.
\end{definition}

\begin{definition} \label{light}
We call $u$ a \emph{light} vertex if $u$ is a 2-vertex, a medium $(1,1,1)$-vertex, or a large $(1,1,1)$-vertex.
\end{definition}

\begin{figure}[H]
\begin{subfigure}[b]{0.32\textwidth}
\centering
\begin{tikzpicture}[scale=0.6]{thick}
\begin{scope}[every node/.style={circle,draw,minimum size=1pt,inner sep=2}]
    \node (1) at (0,0) {};
    \node[fill,label={below:$u$}] (2) at (2,0) {};
    \node (3) at (4,0) {};
\end{scope}

\begin{scope}[every edge/.style={draw=black}]
    \path (1) edge (3);
\end{scope}
\end{tikzpicture}
\caption{A $2$-vertex.}
\end{subfigure}
\begin{subfigure}[b]{0.32\textwidth}
\centering
\begin{tikzpicture}[scale=0.6]{thick}
\begin{scope}[every node/.style={circle,draw,minimum size=1pt,inner sep=2}]
	\node (0) at (-2,0) {3};
    \node[fill] (1) at (0,0) {};
    \node[fill,label={below:$u$}] (2) at (2,0) {};
    \node[fill] (3) at (4,0) {};
    \node (4) at (6,0) {3};
    
    \node[fill] (2') at (2,2) {};
    \node (2'') at (2,4) {4};
\end{scope}

\begin{scope}[every edge/.style={draw=black}]
    \path (0) edge (4);
    \path (2) edge (2'');
\end{scope}
\end{tikzpicture}
\caption{A medium $(1,1,1)$-vertex.}
\end{subfigure}
\begin{subfigure}[b]{0.32\textwidth}
\centering
\begin{tikzpicture}[scale=0.6]{thick}
\begin{scope}[every node/.style={circle,draw,minimum size=1pt,inner sep=2}]
	\node (0) at (-2,0) {4};
    \node[fill] (1) at (0,0) {};
    \node[fill,label={below:$u$}] (2) at (2,0) {};
    \node[fill] (3) at (4,0) {};
    \node (4) at (6,0) {4};
    
    \node[fill] (2') at (2,2) {};
    \node (2'') at (2,4) {3};
\end{scope}

\begin{scope}[every edge/.style={draw=black}]
    \path (0) edge (4);
    \path (2) edge (2'');
\end{scope}
\end{tikzpicture}
\caption{A large $(1,1,1)$-vertex.}
\end{subfigure}
\caption{Light vertices.}
\end{figure}
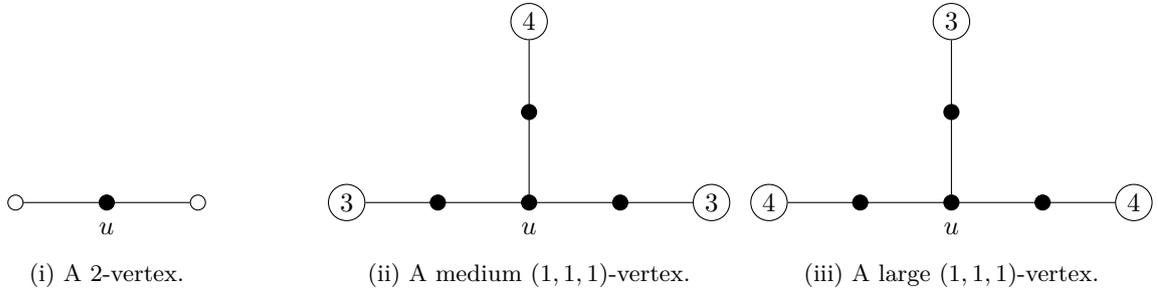

\begin{figure}[H]
\begin{subfigure}[b]{0.49\textwidth}
\centering
\begin{tikzpicture}[scale=0.6]{thick}
\begin{scope}[every node/.style={circle,draw,minimum size=1pt,inner sep=2}]
	\node (0) at (-2,0) {3};
    \node[fill] (1) at (0,0) {};
    \node[fill,label={below:$u$}] (2) at (2,0) {};
    \node[fill] (3) at (4,0) {};
    \node (4) at (6,0) {3};
    
    \node[fill] (2') at (2,2) {};
    \node (2'') at (2,4) {3};
\end{scope}

\begin{scope}[every edge/.style={draw=black}]
    \path (0) edge (4);
    \path (2) edge (2'');
\end{scope}
\end{tikzpicture}
\caption{A small $(1,1,1)$-vertex.}
\end{subfigure}
\begin{subfigure}[b]{0.49\textwidth}
\centering
\begin{tikzpicture}[scale=0.6]{thick}
\begin{scope}[every node/.style={circle,draw,minimum size=1pt,inner sep=2}]
	\node (0) at (-2,0) {4};
    \node[fill] (1) at (0,0) {};
    \node[fill,label={below:$u$}] (2) at (2,0) {};
    \node[fill] (3) at (4,0) {};
    \node (4) at (6,0) {4};
    
    \node[fill] (2') at (2,2) {};
    \node (2'') at (2,4) {4};
\end{scope}

\begin{scope}[every edge/.style={draw=black}]
    \path (0) edge (4);
    \path (2) edge (2'');
\end{scope}
\end{tikzpicture}
\caption{A huge $(1,1,1)$-vertex.}
\end{subfigure}
\begin{subfigure}[b]{\textwidth}
\centering
\begin{tikzpicture}[scale=0.6]{thick}
\begin{scope}[every node/.style={circle,draw,minimum size=1pt,inner sep=2}]
	\node (0) at (-2,0) {};
    \node[fill] (1) at (0,0) {};
    \node[fill,label={below:$u$}] (2) at (2,0) {};
    \node[fill] (3) at (4,0) {};
	\node[fill] (5) at (6,0) {};
	\node[fill] (6) at (8,0) {};
	\node (7) at (10,0) {3};    
    
    \node (2') at (2,2) {3};
    
    \node[fill] (5') at (6,2) {};
    \node (5'') at (6,4) {3};
\end{scope}

\begin{scope}[every edge/.style={draw=black}]
    \path (0) edge (7);
    \path (2) edge (2');
    \path (5) edge (5'');
\end{scope}
\end{tikzpicture}
\caption{A special $(1,1,0)$-vertex.}
\end{subfigure}
\caption{}
\end{figure}

\begin{lemma}\label{3-111-111}
If two $(1,1,1)$-vertices share a common $2$-neighbor, then they must be large $(1,1,1)$-vertices.
\end{lemma}

\begin{proof}
Let $uu_1v$ be a 1-path where both $u$ and $v$ are $(1,1,1)$-vertices. Let $u_2$ and $u_3$ be $u$'s other $2$-neighbors and $v_1$ and $v_2$ be $v$'s other $2$-neighbors. Let $w$ be the other endvertex of $wu_2u$. Since $g(G)\geq 10$, all named vertices are distinct. See \Cref{fig:lemma17}.

Suppose by contradiction that $d(w)=3$, in other words, that $u$ is not a large $(1,1,1)$-vertex. We color $G-\{u,u_1,u_2,u_3\}$ by minimality of $G$ and uncolor $v$. Then, we color $u_3$ and finish with $u_2$, $u$, $u_1$, and $v$ thanks to \Cref{1}.
\end{proof}

\begin{lemma}\label{3-10(-3)1-111-3}
A special $(1,1,0)$-vertex must share a $2$-neighbor with a $4$-vertex.
\end{lemma}

\begin{proof}
Let $u_2u_3u_4$ be a 1-path where $u_2$ is a special $(1,1,0)$-vertex and $u_4$ is a small $(1,1,1)$-vertex. Let $u_1\neq u_3$ be $u_2$'s other $2$-neighbor. See \Cref{fig:lemma18}. Suppose by contradiction that $u_1$ is adjacent to another $3$-vertex. Since $g(G)\geq 10$, all vertices that see $u_3$ are distinct. We color $G-\{u_1,u_2,u_3\}$ by minimality of $G$ and we uncolor $u_4$. We extend the coloring to $u_1$, $u_2$, $u_3$, and $u_4$ thanks to \Cref{1}.
\end{proof}

\begin{figure}[H]
\begin{subfigure}[b]{0.49\textwidth}
\centering
\begin{tikzpicture}[scale=0.6]{thick}
\begin{scope}[every node/.style={circle,draw,minimum size=1pt,inner sep=2}]
	\node (0) at (-2,0) {};
    \node[fill,label={below:$u_3$}] (1) at (0,0) {};
    \node[fill,label={below:$u$}] (2) at (2,0) {};
    \node[fill,label={below:$u_1$}] (3) at (4,0) {};
	\node[fill,label={below:$v$}] (5) at (6,0) {};
	\node[fill,label={below:$v_2$}] (6) at (8,0) {};
	\node (7) at (10,0) {};    
    
    \node[fill,label={left:$u_2$}] (2') at (2,2) {};
    \node[label={above:$w$}] (2'') at (2,4) {};
    
    \node[fill,label={left:$v_1$}] (5') at (6,2) {};
    \node (5'') at (6,4) {};
\end{scope}

\begin{scope}[every edge/.style={draw=black}]
    \path (0) edge (7);
    \path (2) edge (2'');
    \path (5) edge (5'');
\end{scope}
\end{tikzpicture}
\caption{\label{fig:lemma17}Large $3$-vertices case.}
\end{subfigure}
\begin{subfigure}[b]{0.49\textwidth}
\centering
\begin{tikzpicture}[scale=0.6]{thick}
\begin{scope}[every node/.style={circle,draw,minimum size=1pt,inner sep=2}]
	\node (0) at (-2,0) {};
    \node[fill,label={below:$u_1$}] (1) at (0,0) {};
    \node[fill,label={below:$u_2$}] (2) at (2,0) {};
    \node[fill,label={below:$u_3$}] (3) at (4,0) {};
	\node[fill,label={below:$u_4$}] (5) at (6,0) {};
	\node[fill] (6) at (8,0) {};
	\node (7) at (10,0) {3};    
    
    \node (2') at (2,2) {3};
    
    \node[fill] (5') at (6,2) {};
    \node (5'') at (6,4) {3};
\end{scope}

\begin{scope}[every edge/.style={draw=black}]
    \path (0) edge (7);
    \path (2) edge (2');
    \path (5) edge (5'');
\end{scope}
\end{tikzpicture}
\caption{\label{fig:lemma18}Special $(1,1,0)$-vertex case.}
\end{subfigure}
\caption{$3$-vertices case.}
\end{figure}
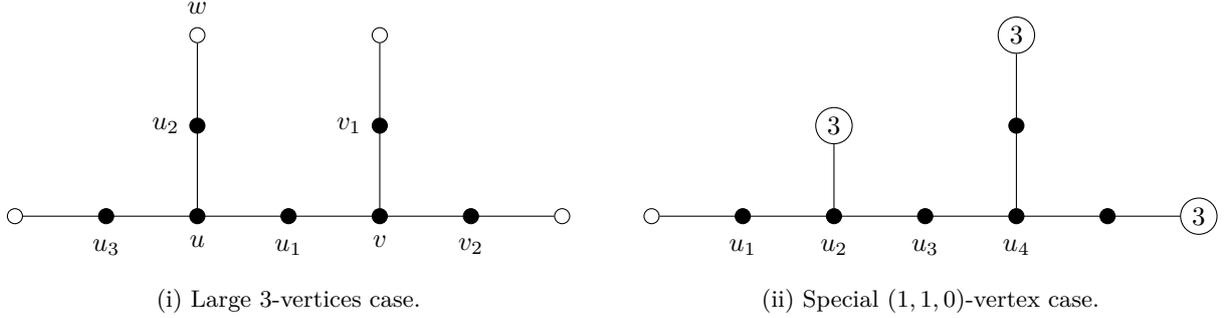

\begin{lemma}\label{4-vertex lemma}
Let $u$ be incident to four $1^+$-paths $uu_iv_i$ for $1\leq i\leq 4$. If $v_1$ and $v_2$ are light vertices, then $d(v_3)\geq 4$ and $d(v_4)\geq 4$.
\end{lemma}

\begin{proof}
Suppose by contradiction that $v_1$ and $v_2$ are light vertices but $d(v_3)\leq 3$. (see \Cref{fig:lemma19}). The proof will proceed as follows. For each combination of light vertices $v_1$ and $v_2$, we will define $H$ a subgraph of $G$. We color $G-H$ by minimality of $G$. Then, let $L(x)$ be the list of remaining colors for every $x\in V(H)$. We will use \Cref{colorable figure} to show that $H$ is always colorable, thus obtaining a valid coloring $G$, which is a contradiction.
Observe that $g(G)\geq 10$ so $g(H)\geq 10$, which means that, in the following subgraphs, every considered vertex will be distinct and their neighborhood at distance at most 2 will be represented exactly by the subgraphs in \Cref{colorable figure}. 

\begin{itemize}
\item If $v_1$ and $v_2$ are $2$-vertices, then $H=\{u,u_1,u_2,u_3,u_4,v_1,v_2\}$ is colorable thanks to \Cref{2}.
\item If $v_1$ is a $2$-vertex and $v_2$ is a medium or large $(1,1,1)$-vertex, then $H=\{u,u_1,u_2,u_3,u_4,v_1,v_2\}\cup N_G(v_2)$ is colorable thanks to \Cref{3}.
\item If $v_1$ and $v_2$ are medium or large $(1,1,1)$-vertices, then $H=\{u,u_1,u_2,u_3,u_4,v_1,v_2\}\cup N_G(v_1)\cup N_G(v_2)$ is colorable thanks to \Cref{4}.
\end{itemize}

By symmetry, the same holds for $d(v_4)$.
\end{proof}

\begin{lemma}\label{3-4-vertex lemma}
Let $u$ be a $4$-vertex with a $3$-neighbor and let $u$ be incident to three $1^+$-paths $uu_iv_i$ for $1\leq i\leq 3$. If $v_1$ and $v_2$ are light vertices, then $v_3$ is a non-special $(1,1,0)$-vertex, a $(1,0,0)$-vertex, or a $4$-vertex.
\end{lemma}

\begin{proof}
Suppose by contradiction that $v_1$ and $v_2$ are light vertices but $v_3$ is a $2$-vertex, a special $(1,1,0)$-vertex, or a $(1,1,1)$-vertex (see \Cref{fig:lemma20}).

We will proceed like the proof of \Cref{4-vertex lemma} by defining a certain subgraph $H$, coloring $G-H$ and extending it to $H$ by using \Cref{colorable figure}.

Let $v_3$ be a special $(1,1,0)$-vertex or a light vertex. When $v_3$ is a $3$-vertex, note that $v_3$ always has a $2$-neighbor $v$ that is adjacent to another $3$-vertex different from $v_3$. We define $N'_G(v_3)=\emptyset$ when $v_3$ is a $2$-vertex and $N'_G(v_3)=\{v\}$ when $v_3$ is a $3$-vertex.
\begin{itemize}
\item If $v_1$ and $v_2$ are $2$-vertices, then $H=\{u,u_1,u_2,u_3,v_1,v_2,v_3\}\cup N'_G(v_3)$ is colorable thanks to \Cref{8} when $N_G'(v_3)=\{v\}$ or \Cref{5} otherwise.
\item If $v_1$ is a $2$-vertex and $v_2$ is a medium or large $(1,1,1)$-vertex, then $H=\{u,u_1,u_2,u_3,v_1,v_2,v_3\}\cup N'_G(v_3)\cup N_G(v_2)$ is colorable thanks to \Cref{9} when $N_G'(v_3)=\{v\}$ or \Cref{6} otherwise.
\item If $v_1$ and $v_2$ are medium or large $(1,1,1)$-vertices, then $H=\{u,u_1,u_2,u_3,v_1,v_2,v_3\}\cup N'_G(v_3)\cup N_G(v_1)\cup N_G(v_2)$ is colorable thanks to \Cref{10} when $N_G'(v_3)=\{v\}$ or \Cref{7} otherwise.
\end{itemize}

Let $v_3$ be a huge $(1,1,1)$-vertex. In the following cases, we actually remove everything in $H$ except $v_3$, color $G-(H\setminus\{v_3\})$, then uncolor $v_3$, and we extend the coloring to $H$.
\begin{itemize}
\item If $v_1$ and $v_2$ are $2$-vertices, then $H=\{u,u_1,u_2,u_3,v_1,v_2,v_3\}$ is colorable thanks to \Cref{5}.
\item If $v_1$ is a $2$-vertex and $v_2$ is a medium or large $(1,1,1)$-vertex, then $H=\{u,u_1,u_2,u_3,v_1,v_2,v_3\}\cup N_G(v_2)$ is colorable thanks to \Cref{6}.
\item If $v_1$ and $v_2$ are medium or large $(1,1,1)$-vertices, then $H=\{u,u_1,u_2,u_3,v_1,v_2,v_3\}\cup N_G(v_1)\cup N_G(v_2)$ is colorable thanks to \Cref{7}.
\end{itemize}
\end{proof}

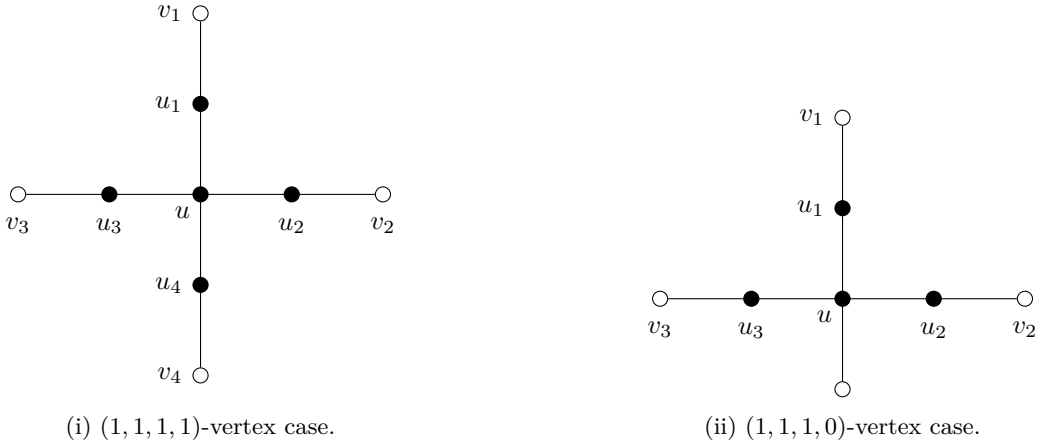
\begin{figure}[H]
\begin{subfigure}[b]{0.49\textwidth}
\centering
\begin{tikzpicture}[scale=0.6]{thick}
\begin{scope}[every node/.style={circle,draw,minimum size=1pt,inner sep=2}]
	\node[label={below:$v_3$}] (0) at (-2,0) {};
    \node[fill,label={below:$u_3$}] (1) at (0,0) {};
    \node[fill,label={below left:$u$}] (2) at (2,0) {};
    \node[fill,label={below:$u_2$}] (3) at (4,0) {};
    \node[label={below:$v_2$}] (4) at (6,0) {};
    
    \node[fill,label={left:$u_1$}] (2') at (2,2) {};
    \node[label={left:$v_1$}] (2'') at (2,4) {};
    
    \node[fill,label={left:$u_4$}] (21') at (2,-2) {};
    \node[label={left:$v_4$}] (21'') at (2,-4) {};
\end{scope}

\begin{scope}[every edge/.style={draw=black}]
    \path (0) edge (4);
    \path (21'') edge (2'');
\end{scope}
\end{tikzpicture}
\caption{\label{fig:lemma19}$(1,1,1,1)$-vertex case.}
\end{subfigure}
\begin{subfigure}[b]{0.49\textwidth}
\centering
\begin{tikzpicture}[scale=0.6]{thick}
\begin{scope}[every node/.style={circle,draw,minimum size=1pt,inner sep=2}]
	\node[label={below:$v_3$}] (0) at (-2,0) {};
    \node[fill,label={below:$u_3$}] (1) at (0,0) {};
    \node[fill,label={below left:$u$}] (2) at (2,0) {};
    \node[fill,label={below:$u_2$}] (3) at (4,0) {};
    \node[label={below:$v_2$}] (4) at (6,0) {};
    
    \node[fill,label={left:$u_1$}] (2') at (2,2) {};
    \node[label={left:$v_1$}] (2'') at (2,4) {};
    
    \node (21') at (2,-2) {};
\end{scope}

\begin{scope}[every edge/.style={draw=black}]
    \path (0) edge (4);
    \path (21') edge (2'');
\end{scope}
\end{tikzpicture}
\caption{\label{fig:lemma20}$(1,1,1,0)$-vertex case.}
\end{subfigure}
\caption{4-vertices case.}
\end{figure}

\subsection{Discharging rules \label{tonton}}

Since $\mad(G)<\frac52$, we must have 
\begin{equation}\label{equation}
\sum_{u\in V(G)} (4d(u)-10) < 0
\end{equation}

We assign to each vertex $u$ the charge $\mu(u)=4d(u)-10$. To prove the non-existence of $G$, we will redistribute the charges preserving their sum and obtaining a non-negative total charge, which will contradict \Cref{equation}.

We then apply the following discharging rules:

\begin{itemize}
\item[\ru0] Every $3^+$-vertex gives 1 to each $2$-vertex on its incident $1^+$-paths.
\item[\ru1] Every $4$-vertex gives 1 to each of its $3$-neighbors.
\item[\ru2] Let $vtu$ be a $1$-path.
\begin{itemize}
\item[(i)] If $v$ is a $(1,1^-,0)$-vertex and $u$ is a small $(1,1,1)$-vertex, then $v$ gives $\frac13$ to $u$.
\item[(ii)] If $v$ is a $4$-vertex and $u$ is a medium $(1,1,1)$-vertex, then $v$ gives 1 to $u$.
\item[(iii)] If $v$ is a $4$-vertex and $u$ is a large $(1,1,1)$-vertex, then $v$ gives $\frac12$ to $u$.
\item[(iv)] If $v$ is a $4$-vertex and $u$ is a huge $(1,1,1)$-vertex, then $v$ gives $\frac13$ to $u$.
\item[(v)] If $v$ is a $4$-vertex and $u$ is a special $(1,1,0)$-vertex, then $v$ gives $\frac13$ to $u$. 
\end{itemize}
\end{itemize}

\begin{figure}[H]
\begin{minipage}[b]{0.65\textwidth}
\centering
\begin{subfigure}[b]{0.32\textwidth}
\centering
\begin{tikzpicture}[scale=0.6]{thick}
\begin{scope}[every node/.style={circle,draw,minimum size=1pt,inner sep=2}]
    \node (1) at (0,0) {$3^+$};
    \node[fill] (2) at (2,0) {};
    \node (3) at (4,0) {$3^+$};
\end{scope}

\begin{scope}[every edge/.style={draw=black}]
    \path (1) edge (3);
    \path[->] (1) edge[bend left] node[above] {1} (2);
    \path[->] (3) edge[bend right] node[above] {1} (2);
\end{scope}
\end{tikzpicture}
\end{subfigure}
\begin{subfigure}[b]{0.32\textwidth}
\centering
\begin{tikzpicture}[scale=0.6]{thick}
\begin{scope}[every node/.style={circle,draw,minimum size=1pt,inner sep=2}]
	\node (0) at (-2,0) {$3^+$};
    \node[fill] (1) at (0,0) {};
    \node[fill] (2) at (2,0) {};
    \node (3) at (4,0) {$3^+$};
\end{scope}

\begin{scope}[every edge/.style={draw=black}]
    \path (0) edge (3);
    \path[->] (0) edge[bend left] node[above] {1} (1);
    \path[->] (0) edge[bend left] node[above] {1} (2);
    \path[->] (3) edge[bend right] node[above] {1} (1);
    \path[->] (3) edge[bend right] node[above] {1} (2);
\end{scope}
\end{tikzpicture}
\end{subfigure}
\caption{\ru0.}
\end{minipage}
\begin{minipage}[b]{0.32\textwidth}
\centering
\begin{tikzpicture}[scale=0.6]{thick}
\begin{scope}[every node/.style={circle,draw,minimum size=1pt,inner sep=2}]
    \node (1) at (0,0) {$4$};
    \node (3) at (4,0) {$3$};
\end{scope}

\begin{scope}[every edge/.style={draw=black}]
    \path (1) edge (3);
    \path[->] (1) edge[bend left] node[above] {1} (3);
\end{scope}
\end{tikzpicture}
\caption{\ru1.}
\end{minipage}
\end{figure}

\begin{figure}[H]
\begin{subfigure}[b]{0.32\textwidth}
\centering
\begin{tikzpicture}[scale=0.6]{thick}
\begin{scope}[every node/.style={circle,draw,minimum size=1pt,inner sep=2}]
	\node[label={below:$v_1$}] (0) at (-2,0) {3};
    \node[fill,label={below:$t_1$}] (1) at (0,0) {};
    \node[fill,label={below:$u$}] (2) at (2,0) {};
    \node[fill,label={below:$t_2$}] (3) at (4,0) {};
    \node[label={below:$v_2$}] (4) at (6,0) {3};
    
    \node[fill,label={left:$t_3$}] (2') at (2,2) {};
    \node[label={above:$v_3$}] (2'') at (2,4) {3};
\end{scope}

\begin{scope}[every edge/.style={draw=black}]
    \path (0) edge (4);
    \path (2) edge (2'');
    \path[->] (0) edge[bend left] node[above] {$\frac13$} (2);
    \path[->] (4) edge[bend right] node[above] {$\frac13$} (2);
    \path[->] (2'') edge[bend left] node[right] {$\frac13$} (2);
\end{scope}
\end{tikzpicture}
\caption{Small $(1,1,1)$-vertex case: \\for $1\leq i\leq 3$, $v_i\neq (1,1,1)$-vertex. }
\end{subfigure}
\begin{subfigure}[b]{0.32\textwidth}
\centering
\begin{tikzpicture}[scale=0.6]{thick}
\begin{scope}[every node/.style={circle,draw,minimum size=1pt,inner sep=2}]
	\node (0) at (-2,0) {3};
    \node[fill] (1) at (0,0) {};
    \node[fill,label={below:$u$}] (2) at (2,0) {};
    \node[fill] (3) at (4,0) {};
    \node (4) at (6,0) {3};
    
    \node[fill,label={left:$t$}] (2') at (2,2) {};
    \node[label={above:$v$}] (2'') at (2,4) {4};
\end{scope}

\begin{scope}[every edge/.style={draw=black}]
    \path (0) edge (4);
    \path (2) edge (2'');
    \path[->] (2'') edge[bend left] node[right] {1} (2);
\end{scope}
\end{tikzpicture}
\caption{Medium $(1,1,1)$-vertex case.}
\end{subfigure}
\begin{subfigure}[b]{0.32\textwidth}
\centering
\begin{tikzpicture}[scale=0.6]{thick}
\begin{scope}[every node/.style={circle,draw,minimum size=1pt,inner sep=2}]
\node[label={below:$v_1$}] (0) at (-2,0) {4};
    \node[fill,label={below:$t_1$}] (1) at (0,0) {};
    \node[fill,label={below:$u$}] (2) at (2,0) {};
    \node[fill,label={below:$t_2$}] (3) at (4,0) {};
    \node[label={below:$v_2$}] (4) at (6,0) {4};
    
    \node[fill] (2') at (2,2) {};
    \node (2'') at (2,4) {3};
\end{scope}

\begin{scope}[every edge/.style={draw=black}]
    \path (0) edge (4);
    \path (2) edge (2'');
    \path[->] (0) edge[bend left] node[above] {$\frac12$} (2);
    \path[->] (4) edge[bend right] node[above] {$\frac12$} (2);
\end{scope}
\end{tikzpicture}
\caption{Large $(1,1,1)$-vertex case.}
\end{subfigure}
\begin{subfigure}[b]{0.32\textwidth}
\centering
\begin{tikzpicture}[scale=0.6]{thick}
\begin{scope}[every node/.style={circle,draw,minimum size=1pt,inner sep=2}]
	\node[label={below:$v_1$}] (0) at (-2,0) {4};
    \node[fill,label={below:$t_1$}] (1) at (0,0) {};
    \node[fill,label={below:$u$}] (2) at (2,0) {};
    \node[fill,label={below:$t_2$}] (3) at (4,0) {};
    \node[label={below:$v_2$}] (4) at (6,0) {4};
    
    \node[fill,label={left:$t_3$}] (2') at (2,2) {};
    \node[label={above:$v_3$}] (2'') at (2,4) {4};
\end{scope}

\begin{scope}[every edge/.style={draw=black}]
    \path (0) edge (4);
    \path (2) edge (2'');
    \path[->] (0) edge[bend left] node[above] {$\frac13$} (2);
    \path[->] (4) edge[bend right] node[above] {$\frac13$} (2);
    \path[->] (2'') edge[bend left] node[right] {$\frac13$} (2);
\end{scope}
\end{tikzpicture}
\caption{Huge $(1,1,1)$-vertex case.}
\end{subfigure}
\begin{subfigure}[b]{0.65\textwidth}
\centering
\begin{tikzpicture}[scale=0.6]{thick}
\begin{scope}[every node/.style={circle,draw,minimum size=1pt,inner sep=2}]
	\node[label={below:$v$}] (0) at (-2,0) {4};
    \node[fill,label={below:$t$}] (1) at (0,0) {};
    \node[fill,label={below:$u$}] (2) at (2,0) {};
    \node[fill] (3) at (4,0) {};
	\node[fill] (5) at (6,0) {};
	\node[fill] (6) at (8,0) {};
	\node (7) at (10,0) {3};    
    
    \node (2') at (2,2) {3};
    
    \node[fill] (5') at (6,2) {};
    \node (5'') at (6,4) {3};
\end{scope}

\begin{scope}[every edge/.style={draw=black}]
    \path (0) edge (7);
    \path (2) edge (2');
    \path (5) edge (5'');
    \path[->] (0) edge[bend left] node[above] {$\frac13$} (2);
\end{scope}
\end{tikzpicture}
\caption{Special $(1,1,0)$-vertex case.}
\end{subfigure}
\caption{\ru2.}
\end{figure}

\subsection{Verifying that charges on each vertex are non-negative} \label{verification}

Let $\mu^*$ be the assigned charges after the discharging procedure. In what follows, we will prove that: $$\forall u \in V(G), \mu^*(v)\geq 0.$$

Let $u$ be a vertex in $V(G)$.\\ 
\textbf{Case 1:} If $d(u) = 2$, then $u$ receives charge 1 from each endvertex of the path it lies on by \ru0. Thus, we get: 
$$\mu^*(u) = \mu(u) + 2\cdot 1 = 4\cdot 2 -10 + 2 = 0.$$

\textbf{Case 2:} If $d(u)=3$, then recall that $\mu(u)=4\cdot 3 - 10 = 2$. Moreover, $u$ cannot be incident to any $2^+$-paths due to \Cref{2-path lemma,3-path lemma}. Now, we distinguish the following cases:
\begin{itemize}
\item If $u$ is a $(1,1,1)$-vertex, then $u$ only gives charge to its $2$-neighbors, more precisely 1 to each of its $2$-neighbors by \ru0. Let $v$, $w$, and $x$ be the other endvertices of the $1$-paths incident to $u$.

If $u$ is a small $(1,1,1)$-vertex, then observe that $v$, $w$, and $x$ are all $(1,1^-,0)$-vertices as they cannot be $(1,1,1)$-vertices due to \Cref{3-111-111}. As a result, $u$ receives $\frac13$ from each of $v$, $w$, and $x$ by \ru2(i). Hence,
$$ \mu^*(u)=2-3\cdot 1 + 3\cdot \frac13 = 0.$$

If $u$ is a medium $(1,1,1)$-vertex, then $u$ receives 1 from one of $v$, $w$, and $x$ by \ru2(ii). Hence,
$$ \mu^*(u)=2-3\cdot 1 + 1 = 0.$$

If $u$ is a large $(1,1,1)$-vertex, then $u$ receives $\frac12$ twice from $v$, $w$, and $x$ by \ru2(iii). Hence,
$$ \mu^*(u)=2-3\cdot 1 + 2\cdot \frac12 = 0.$$

If $u$ is a huge $(1,1,1)$-vertex, then $u$ receives $\frac13$ from each of $v$, $w$, and $x$ by \ru2(iv). Hence,
$$ \mu^*(u)=2-3\cdot 1 + 3\cdot \frac13 = 0.$$

\item If $u$ is a $(1,1,0)$-vertex, then $u$ gives 1 to each of its $2$-neighbors by \ru0. Let $t$ be its $3^+$-neighbor and let $v$ and $w$ be the other endvertices of the $1$-paths incident to $u$. First, observe that if neither $v$ nor $w$ is a small $(1,1,1)$-vertex, then we would have
$$ \mu^*(u)\geq 2-2\cdot 1 = 0.$$
So, say $v$ is a small $(1,1,1)$-vertex, in which case, $u$ gives $\frac13$ to $v$ by \ru2(i).

If $d(t)=4$, then $u$ receives 1 from $t$ by \ru1. Moreover, at worst, $u$ also gives $\frac13$ to $w$ by \ru2(i). To sum up,
$$ \mu^*(u)\geq 2 - 2\cdot 1 + 1 - 2\cdot \frac13 = \frac13.$$

If $d(t)=3$, then $w$ must be a $4$-vertex by \Cref{3-10(-3)1-111-3}. In other words, $u$ is a special $(1,1,0)$-vertex. Thus, $u$ receives $\frac13$ from $w$ by \ru3. To sum up,
$$ \mu^*(u)\geq 2 - 2\cdot 1 -\frac13 + \frac13 = 0.$$

\item If $u$ is a $(1,0,0)$-vertex, then at worst, $u$ gives 1 to its 2-neighbor by \ru0 and $\frac13$ to the other endvertex of its incident 1-path by \ru2(i). Thus,
$$ \mu^*(u)\geq 2 - 1 - \frac13 = \frac23.$$

\item If $u$ is a $(0,0,0)$-vertex, then none of the discharging rules apply. Thus,
$$ \mu^*(u) = \mu(u) = 2.$$
\end{itemize}

\textbf{Case 3: } If $d(u)=4$, then recall that $\mu(u) = 4\cdot 4 - 10 = 6$. Observe that $u$ gives away at most 2 per incident $0^+$-path. Indeed, there are no $3^+$-paths by \Cref{3-path lemma}. More precisely, $u$ gives:
\begin{itemize}
\item[2] to a $2$-path by \ru0.
\item[2] to a $1$-path with a medium $(1,1,1)$-endvertex by \ru0 and \ru2(ii).
\item[$\frac32$] to a $1$-path with a large $(1,1,1)$-endvertex by \ru0 and \ru2(iii).
\item[$\frac43$] to a $1$-path with a huge $(1,1,1)$-endvertex by \ru0 and \ru2(iv).
\item[$\frac43$] to a $1$-path with a special $(1,1,0)$-endvertex by \ru0 and \ru2(v).
\item[1] to a $1$-path with a non-special $(1,1,0)$-endvertex, a $(1,0,0)$-endvertex or a $4$-endvertex by \ru0.
\item[1] to a $3$-neighbor by \ru1.
\end{itemize}

Observe that $u$ only gives more than $\frac43$ to an incident path when the neighbor at distance $2$ on that path is a light vertex (\Cref{light}). Now, we distinguish the following cases:

\begin{itemize}
\item If $u$ is a $(1^+,1^+,1^+,1^+)$-vertex, then let $uu_iv_i$ be the $1^+$-paths incident to $u$ for $1\leq i\leq 4$.

If at most one of the $v_i$'s is a light vertex, then at worst
$$ \mu^*(u)\geq 6 - 2 - 3\cdot\frac43 = 0.$$

If at least two of the $v_i$'s are light vertices, say $v_1$ and $v_2$, then $d(v_3)\geq 4$ and $d(v_4)\geq 4$ due to \Cref{4-vertex lemma}. As a result, $u$ gives only 1 to each of $uu_3v_3$ and $uu_4v_4$. Thus, at worst we get
$$ \mu^*(u)\geq 6 - 2\cdot 2 - 2\cdot 1 = 0.$$

\item If $u$ is a $(1^+,1^+,1^+,0)$-vertex with a $4$-neighbor, then $u$ does not give anything to its $4$-neighbor. Thus, at worst we have
$$ \mu^*(u)\geq 6 - 3\cdot 2 = 0.$$

\item If $u$ is a $(1^+,1^+,1^+,0)$-vertex with a $3$-neighbor, then let $uu_iv_i$ be the $1^+$-paths incident to $u$ for $1\leq i\leq 3$. We know that $u$ always give 1 to its $3$-neighbor.

If at most one of the $v_i$'s is a light vertex, then at worst
$$ \mu^*(u)\geq 6 - 1 - 2 - 2\cdot\frac43 = \frac13.$$

If at least two of the $v_i$'s are light vertices, say $v_1$ and $v_2$, then $v_3$ must be a non-special $(1,1,0)$-endvertex, a $(1,0,0)$-endvertex or a $4$-endvertex due to \Cref{3-4-vertex lemma}. As a result, $u$ gives only 1 to $uu_3v_3$. Thus, at worst we get
$$ \mu^*(u)\geq 6 - 1 - 2\cdot 2 - 1 = 0.$$

\item If $u$ is a $(0^+,0^+,0,0)$-vertex, then at worst we have
$$ \mu^*(u)\geq 6 - 2\cdot 2 -2\cdot 1 = 0.$$
\end{itemize}

To conclude, we started with a charge assignment with a negative total sum, but after the discharging procedure, which preserved that sum, we end up with a non-negative one, which is a contradiction. In other words, there exists no counter-example to \Cref{main theorem}. 

\subsection*{Remarks}

Graphs of with $mad<\frac52$ contains all planar graphs with girth at least 10. The condition on the girth in \Cref{main theorem} only serves to simplify the proof of the reducibility of certain configurations as it guarantees that some vertices must be distinct. Here, we chose girth 10 as it would coincide with the girth of the subclass of planar graphs, but girth 9 suffices to guarantee the desired property.  One can also strengthen \Cref{main theorem} by removing this condition and study the cases where some vertices in our configurations coincide.

\section*{Acknowledgements}
This work was partially supported by the grant HOSIGRA funded by the French National Research
Agency (ANR, Agence Nationale de la Recherche) under the contract number ANR-17-CE40-0022.

\bibliographystyle{plain}

\end{document}